\documentclass[12pt,reqno]{amsart}
\usepackage{mathrsfs}
\usepackage[breaklinks]{hyperref}
\usepackage[dvipsnames]{xcolor} 
\usepackage{diagbox}
\usepackage{latexsym}
\usepackage{upref, eucal}
\usepackage{tikz}
\usepackage{amsmath, amssymb, graphicx}
\usepackage{epsfig}
\usepackage{tkz-euclide}

\usepackage[headheight=110pt,top=1in, bottom=.9in, left=0.8in, right=0.8in]{geometry}
\allowdisplaybreaks
\usepackage{url}
\usepackage{amssymb}

\renewcommand{\(}{\left\(}
\renewcommand{\)}{\right\)}
\renewcommand{\[}{\left\[}
\renewcommand{\]}{\right\]}

\renewcommand{\i}{\infty}
\renewcommand{\pmod}[1]{\,(\textup{mod}\,#1)}
\numberwithin{equation}{section}
\theoremstyle{plain}
\newtheorem{theorem}{Theorem}[section]

\newtheorem{corollary}[theorem]{Corollary}

\newtheorem{remark}[]{Remark}

\makeatletter
\def\proof{\@ifnextchar[{\@oproof}{\@nproof}}
\def\@oproof[#1][#2]{\trivlist\item[\hskip\labelsep\textit{#2 Proof of\
		#1.}~]\ignorespaces}
\def\@nproof{\trivlist\item[\hskip\labelsep\textit{Proof.}~]\ignorespaces}

\makeatother

\makeatletter
\def\@tocline#1#2#3#4#5#6#7{\relax
	\ifnum #1>\c@tocdepth 
	\else
	\par \addpenalty\@secpenalty\addvspace{#2}%
	\begingroup \hyphenpenalty\@M
	\@ifempty{#4}{%
		\@tempdima\csname r@tocindent\number#1\endcsname\relax
	}{%
		\@tempdima#4\relax
	}%
	\parindent\z@ \leftskip#3\relax \advance\leftskip\@tempdima\relax
	\rightskip\@pnumwidth plus4em \parfillskip-\@pnumwidth
	#5\leavevmode\hskip-\@tempdima
	\ifcase #1
	\or\or \hskip 1em \or \hskip 2em \else \hskip 3em \fi%
	#6\nobreak\relax
	\dotfill\hbox to\@pnumwidth{\@tocpagenum{#7}}\par
	\nobreak
	\endgroup
	\fi}
\makeatother

\allowdisplaybreaks

\begin{document}

\title[The Rogers-Ramanujan dissection of a theta function]{The Rogers-Ramanujan dissection of a theta function}
\author{Atul Dixit and Gaurav Kumar}
\address{Department of Mathematics, Indian Institute of Technology Gandhinagar, Palaj, Gandhinagar 382355, Gujarat, India} 
\email{adixit@iitgn.ac.in; kumargaurav@iitgn.ac.in}
	\begin{abstract}
		Page 27 of Ramanujan's Lost Notebook contains a beautiful identity which not only gives, as a special case, a famous modular relation between the Rogers-Ramanujan functions $G(q)$ and $H(q)$ but also a relation between two fifth order mock theta functions and $G(q)$ and $H(q)$. We generalize Ramanujan's relation with the help of a parameter $s$ to get an infinite family of such identities. Our result shows that a theta function can always be ``dissected'' as a finite sum of products of generalized Rogers-Ramanujan functions. Several well-known results are shown to be consequences of our theorem, for example, a generalization of the Jacobi triple product identity and Andrews' relation between two of his generalized third order mock theta functions. We give enough evidence, through asymptotic analysis as well as by other means, to show that the identities we get from our main result for $s>2$ transcend the modular world and hence look difficult to be written in the form of a modular relation. Using asymptotic analysis, we also offer a clinching evidence that explains how Ramanujan may have arrived at his generalized modular relation.
	\end{abstract}
\maketitle
\tableofcontents

\section{Introduction}\label{intro}

Ramanujan's Lost Notebook \cite{lnb} contains several beautiful gems which leave us wondering how one could ever conceive them. An amusing collection of ten such results can be found in \cite{hit parade}.

Although not a part of the list in \cite{hit parade}, there is another gem found on page $27$ of  \cite{lnb}. It is a generalized modular relation which states that for $a\in\mathbb{C}\backslash\{0\}$, and $b\in\mathbb{C}$,
\begin{align}\label{gmr}
&\sum_{m=0}^{\infty}\frac{a^{-2m}q^{m^2}}{(bq)_m}\sum_{n=0}^{\infty}\frac{a^nb^nq^{{n^2}/4}}{(q)_n}+\sum_{m=0}^{\infty}\frac{a^{-2m-1}q^{m^2+m}}{(bq)_m}\sum_{n=0}^{\infty}\frac{a^nb^nq^{{(n+1)^2}/4}}{(q)_n}\nonumber\\
&=\frac{1}{(bq)_{\infty}}\sum_{n=-\infty}^{\infty}a^nq^{n^2/4}-(1-b)\sum_{n=1}^{\infty}a^nq^{n^2/4}\sum_{\ell=0}^{n-1}\frac{b^\ell}{(q)_\ell}.
\end{align}
Here, and throughout the paper, we use the standard $q$-series notation
\begin{align*}
	(A)_0 &:=(A;q)_0 =1, \qquad \\
	(A)_n &:=(A;q)_n  = (1-A)(1-Aq)\cdots(1-Aq^{n-1}),
	\qquad n \geq 1, \\
	(A)_{\infty} &:=(A;q)_{\i}  = \lim_{n\to\i}(A;q)_n, \qquad |q|<1.
\end{align*}
What is so striking about \eqref{gmr}? First of all, if we let $a=b=1$, replace $q$ by $q^4$, multiply both sides of the resulting identity by $(-q^2;q^2)_{\infty}$, and use Rogers' identities \cite[pp.~330-331]{rogers2}, namely,
\begin{align}
G(q)&=(-q^2;q^2)_{\infty}\sum_{n=0}^{\infty}\frac{q^{n^2}}{(q^4;q^4)_n},\label{rogers identities1}\\
H(q)&=(-q^2;q^2)_{\infty}\sum_{n=0}^{\infty}\frac{q^{n^2+2n}}{(q^4;q^4)_n},\label{rogers identities2}
\end{align}
where
\begin{align}\label{rrf defn}
G(q):=\sum_{n=0}^{\infty}\frac{q^{n^2}}{(q)_n}\hspace{2mm}\text{and}\hspace{2mm} H(q):=\sum_{n=0}^{\infty}\frac{q^{n^2+n}}{(q)_n}\end{align}
are the Rogers-Ramanujan functions, then one obtains
\begin{equation}\label{mre}
G(q)G(q^4)+qH(q)H(q^4)=\frac{\varphi(q)}{(q^2;q^2)_{\infty}}=(-q;q^2)_{\infty}^{2},
\end{equation}
where 
\begin{equation}\label{varphiq}
\varphi(q):=\sum_{n=-\infty}^{\infty}q^{n^2}.
\end{equation}
 Thus the well-known modular relation \eqref{mre}, also due to Ramanujan \cite[p.~77, Equation (2)]{birch}, \cite[p.~7]{berndt40}, is but a special case of \eqref{gmr}! Hence \eqref{gmr} is called a \emph{generalized modular relation}; see \cite[p.~150]{aar2}. Regarding \eqref{mre}, Andrews \cite[p.~xxi]{lnb} says,\\

\textit{This sort of identity has always appeared to me to lie totally within the realm of modular functions and to be completely resistant to $q$-series generalization. One of the greatest shocks I got from the Lost Notebook was the following assertion...}\\

The assertion Andrews is referring to in the above quote is nothing but \eqref{gmr} with $b=1$ and $q$ replaced by $q^{4}$.

A second striking feature of \eqref{gmr} is that letting $a=b=-1$ in \eqref{gmr}, replacing $q$ by $q^4$, multiplying the resulting two sides by $(-q^2;q^2)_{\infty}$, and then using \eqref{rogers identities1} and \eqref{rogers identities2} yields \cite[p.~28, Equation (2.3.11)]{yesto}\footnote{There are two minor slips in the version stated in \cite{yesto}, namely, the right-hand side should be multiplied by $(-q^2;q^2)_{\infty}$ and the $n$ in the $q-$product in the denominator of the double sum should be $j$.}
\begin{align}\label{gh5mock}
G(q)f_0(q^4)-qH(q)f_1(q^4)=(-q^2;q^4)_{\infty}\sum_{n=-\infty}^{\infty}(-1)^nq^{n^2}+2(-q^2;q^2)_{\infty}\sum_{n, \ell\geq0}\frac{(-1)^nq^{(n+\ell+1)^2}}{(q^4;q^4)_\ell},
\end{align} 
where $f_0(q):=\sum_{n=0}^{\infty}q^{n^2}/(-q;q)_n$ and $f_1(q):=\sum_{n=0}^{\infty}q^{n^2+n}/(-q;q)_n$ are two fifth order mock theta functions of Ramanujan. As rightly pointed out by Andrews \cite[p.~28]{yesto}, there isn't any result of this type in any of the literature on mock theta functions. 

Thus, both \eqref{mre} and \eqref{gh5mock} are special cases of \eqref{gmr}! Before Ramanujan found \eqref{gmr}, there was no precedent for anything of this kind. 

Thirdly, if we let just $b=1$ in \eqref{gmr}, then the second expression on the right-hand side vanishes, and thus one sees that the Jacobi theta function is essentially, that is, modulo the $q$-product $1/(q)_{\infty}$, a sum of products of two generalized Rogers-Ramanujan functions! By a \emph{generalized Rogers-Ramanujan function}, we mean a series of the form
\begin{align}\label{grrf}
\sum_{n=0}^{\infty}\frac{a^{n}q^{(cn^2+dn)/s}}{(bq)_n},
\end{align}
where $a,b\in\mathbb{C}, a\neq0$, $c, d\in\mathbb{R}$, and $s\in\mathbb{N}$.

Andrews \cite[Chapter II]{yesto} gave a detailed account of these identities as well as the aforementioned history surrounding them. In the same article, he proved \eqref{gmr} by showing that the coefficients of $a^{N}, -\infty<N<\infty$, on both sides are identical. See also \cite[p.~150, Section 7.2]{aar2}, where this proof was reproduced. However, this requires knowing the identity in advance.

The current work has its genesis in trying to obtain a natural proof of \eqref{gmr}. Before proceeding in that direction though, it is worthy to mention that on page $26$ of the Lost Notebook, Ramanujan gave the following claim, which, upon changing to the modern notation, reads
\begin{align}\label{asymptotic of product}
	\sum_{n=0}^{\infty}\frac{a^nq^{n^2/(2s)}}{(q)_n}\sum_{n=0}^{\infty}\frac{a^{-ns}	q^{n^2s/2}}{(q)_n}\hspace{7mm}\text{as}\hspace{1mm}q\to1??
\end{align}
When $s=2$, the above product is nothing but the first expression we get on the resulting left-hand side of \eqref{gmr} after letting $b=1$.	Regarding \eqref{asymptotic of product}, Andrews and Berndt \cite[p.~156]{aar2} say, \emph{`Ramanujan provides no indication either why this is of interest for arbitrary $s$ or what the asymptotics should be.'}

In Theorem \ref{asymptotic of product theorem} below, we derive the following asymptotic formula for the product in \eqref{asymptotic of product} for $a>0$, namely, as $q\to 1^{-}$,
\begin{align*}
\sum_{n=0}^{\infty}\frac{a^nq^{n^2/(2s)}}{(q)_n}\sum_{n=0}^{\infty}\frac{a^{-ns}	q^{n^2s/2}}{(q)_n}\sim\frac{\sqrt{s}}{1+(s-1)z_1}\exp{\left(-\frac{1}{\log(q)}\left(\frac{\pi^2}{6}+\frac{s}{2}\log^{2}(a)\right)\right)},
\end{align*}
where $z_1$ is the real positive root of $az^{1/s}+z-1=0$. As explained in the paragraphs preceding Theorem \ref{asymptotic of product theorem} below, the asymptotic formula is elegant, and this is precisely because we have taken the \emph{product} of the two series rather than just one of them. 

However, the fact that the special case $s=2$ of the product in \eqref{asymptotic of product} is the first term on the left-hand side of \eqref{gmr} (with $b=1$) begs to ask the following question - \emph{is there a generalization of \eqref{gmr} for any $s\in\mathbb{N}$ in which the left-hand side of \eqref{asymptotic of product}, generalized by means of a parameter $b$, occurs as the first term of the corresponding generalization of the left-hand side of \eqref{gmr}?} Indeed, the main objective of this paper is to answer this question in the affirmative and then derive a plethora of important and beautiful identities resulting from such a generalization. This shows that \eqref{gmr} is not an isolated result, rather, it is one among an infinite family of representations for the Jacobi theta function!

Our main theorem is as follows.
\begin{theorem}\label{theorem2}
	Let $s\in\mathbb{N}$.	For $a\in\mathbb{C}$, $a\neq 0$, and $b\in\mathbb{C}$,
	\begin{align}\label{theorem 2 eqn}
		\sum_{k=0}^{s-1}\bigg\{\sum_{m=0}^{\infty}\frac{a^{-sm-k}q^{(sm+k)^2/(2s)}}{(bq)_{m}} &\sum_{n=0}^{\infty}\frac{a^{n}b^{n}q^{n(n+2js-2k)/(2s)}}{(q)_{n}} \bigg\} \notag \\ &=
		\frac{1}{(bq)_{\infty}} \sum_{n=-\infty}^{\infty}a^{n}q^{n^2/(2s)} - (1-b) \sum_{n=1}^{\infty}a^{n}q^{n^2/(2s)}  \sum_{\ell=0}^{n-1}\frac{b^\ell}{(q)_{\ell}},
	\end{align}
		where
	\begin{align}\label{j}
		j= \left\{
		\begin{array}{ll}
			0,  & \mbox{if } k = 0, \\
			1, & \mbox{otherwise. }
		\end{array}
		\right. 
	\end{align}
\end{theorem}
The way we prove this result is by obtaining identities for each of the partial theta functions $\sum_{n=1}^{\infty}a^nq^{n^2/(2s)}$ and $\sum_{n=-\infty}^{0}a^nq^{n^2/(2s)}$ and then adding the corresponding sides of these identities. See Theorems \ref{theorem2-1} and \ref{theorem2-2} in Section \ref{pmr}.

There are a number of formulas in the literature which express a product of theta function with a $q$-series or another theta function as a finite sum of theta functions or their products or powers, for example, Schr\"{o}ter's formula \cite[Theorem 1]{liu-schroter}, Ramanujan's circular summation of theta functions \cite[p.~54]{lnb}, \cite[Equation (1.3)]{son}. 

However, the one we have derived in Theorem \ref{theorem2} is of a different kind wherein the summand involves generalized Rogers-Ramanujan functions of the form \eqref{grrf}. Since letting $b=1$ in \eqref{theorem 2 eqn} leaves us with only the theta function on its right-hand side, we call such a representation\footnote{Note that we have allowed the $q$-product $(q)_{\infty}$ to be present in the definition of the Rogers-Ramanujan dissection.} of the theta function in terms of the generalized Rogers-Ramanujan functions as its \emph{Rogers-Ramanujan dissection}.
 
The proof of Theorem \ref{theorem2} is quite technical and requires quite a bit of bookkeeping as well. It is reserved for Section \ref{pmr}.  Next, we give some of corollaries resulting from Theorem \ref{theorem2}. 

\begin{corollary}\label{ramanujan27}
Ramanujan's identity \eqref{gmr} holds.
\end{corollary}

Theorem \ref{theorem2} gives, as another special case, a nice generalization of the Jacobi triple identity (see \eqref{jtpi} below), which, to the best of our knowledge, has not appeared in print before.
\begin{corollary}\label{gen jtpi}
	For $a\in\mathbb{C}$, $a\neq 0$, and $b\in\mathbb{C}$,
\begin{align*}
	&\sum_{m=0}^{\infty}\frac{a^{-m}q^{m^2/2}}{(bq)_m}\sum_{n=0}^{\infty}\frac{a^nb^nq^{{n^2}/2}}{(q)_n}=\frac{1}{(bq)_{\infty}} \sum_{n=-\infty}^{\infty}a^{n}q^{n^2/2} - (1-b) \sum_{n=1}^{\infty}a^{n}q^{n^2/2} \sum_{\ell=0}^{n-1}\frac{b^\ell}{(q)_{\ell}}.
	\end{align*}
\end{corollary}
A beautiful identity for the product of partial theta functions, as opposed to the generalized Rogers-Ramanujan functions occurring in the above identity, and involving three variables $a, b$ and $q$, was obtained by Andrews and Warnaar \cite[Theorem 1.1]{andrews-warnaar}. An interesting generalization of this identity was recently given by Chan, Chan, Chen and Huang \cite[Theorem 1]{chan fine}.\\

\noindent
The special case $s=3$ of Theorem \ref{theorem2} leads to
\begin{corollary}\label{s=3case}
		For $a\in\mathbb{C}\backslash\{0\}$ and $b\in\mathbb{C}$,
\begin{align}\label{s=3mre}
	&\sum_{m=0}^{\infty}\frac{a^{-3m}q^{3m^2/2}}{(bq)_m}\sum_{n=0}^{\infty}\frac{a^nb^nq^{n^2/6}}{(q)_n}+q^{1/6}\sum_{m=0}^{\infty}\frac{a^{-3m-1}q^{3m^2/2+m}}{(bq)_m}\sum_{n=0}^{\infty}\frac{a^nb^nq^{n(n+4)/6}}{(q)_n}\nonumber\\
	&+q^{2/3}\sum_{m=0}^{\infty}\frac{a^{-3m-2}q^{3m^2/2+2m}}{(bq)_m}\sum_{n=0}^{\infty}\frac{a^nb^nq^{n(n+2)/6}}{(q)_n}\nonumber\\
	&=	\frac{1}{(bq)_{\infty}} \sum_{n=-\infty}^{\infty}a^{n}q^{n^2/6} - (1-b) \sum_{n=1}^{\infty}a^{n}q^{n^2/6} \sum_{\ell=0}^{n-1}\frac{b^\ell}{(q)_{\ell}}.
\end{align}	
\end{corollary}
In light of the fact that \eqref{mre} can be derived from \eqref{gmr} by letting $a=b=1$ in the latter, the natural thing to try is to see if letting $a=b=1$ in \eqref{s=3mre}, or, for that matter, in \eqref{theorem 2 eqn}, produces something analogous. The key ingredients in obtaining \eqref{mre} from \eqref{gmr} were Rogers' identities  \eqref{rogers identities1} and \eqref{rogers identities2}. However, it is highly unlikely that analogues of these identities exist for $s>2$. A detailed asymptotic analysis done in Section \ref{cr} for $s=3$ fails to point us to a relation of the type in \eqref{rogers identities1}. We now give another reason why such a relation seems implausible for $s>2$. Observe that \eqref{rogers identities1} is, in fact, a special case of Rogers' more general identity \cite[Equation (3)]{bressoud1}
\begin{align*}
\sum_{n=0}^{\infty}\frac{q^{n^2}a^n}{(q)_n}=(-aq^2;q^2)_{\infty}\sum_{n=0}^{\infty}\frac{q^{n^2}a^n}{(q^2;q^2)_n(-aq^2;q^2)_n}.
\end{align*}
Now Bressoud \cite[Equation (5)]{bressoud1} has further generalized the above identity by obtaining, for any $s\in\mathbb{N}$,
\begin{align}\label{bressoud generalization}
\sum_{m=0}^{\infty}\frac{q^{m+sm(m-1)/2}a^m}{(q)_m}=(-aq^s;q^s)_{\infty}\sum_{n_1,\cdots,n_{s-1}\geq0}\frac{a^Nq^{sN(N-1)/2+n_1+2n_2+\cdots+(s-1)n_{s-1}}}{(q^s;q^s)_{n_1}\cdots(q^s;q^s)_{n_{s-1}}(-aq^s;q^s)_N},
\end{align}
where $N=n_1+n_2+\cdots+n_{s-1}$. Observe that if we let $s=3$ and $a=\sqrt{q}$ in the above identity, then its resulting left-hand side, that is, $\sum_{m=0}^{\infty} q^{3m^2/2}/(q)_m$, is nothing but the first series in the first expression\footnote{The series with index of summation $m$ in the other two expressions, and with $a=b=1$, can also be shown to be special cases of the left-hand side of \eqref{bressoud generalization} upon appropriate substitutions.} on the left-hand side of \eqref{s=3mre} with $a=b=1$. However, the appearance of the multi-sum, for $s>2$, on the right-hand side of \eqref{bressoud generalization} presents difficulty in having $\sum_{m=0}^{\infty} q^{3m^2/2}/(q)_m$ related to the second series in that expression with $a=b=1$, that is, to $\sum_{n=0}^{\infty}q^{n^2/6}/(q)_n$, or, to $\sum_{n=0}^{\infty}q^{cn^2/6}/(q^c;q^c)_n$ for some $c\in\mathbb{N}$.

 We note in passing that Bressoud, Santos and Mondek \cite[Theorem 1]{bressoud2} have shown that the special case $a=q^{i-1}, 1\leq i\leq k$, of the left-hand side of \eqref{bressoud generalization} is the generating function of three restricted partition functions, and Lehmer \cite{lehmer} has shown that there is no nice infinite product representation for this generating function for $s>2$ as opposed to the sums with $s=2$ and $i=1, 2$ having infinite product representations which constitute the Rogers-Ramanujan identities.

Another interesting corollary of Theorem \ref{theorem2} which we derive is a result of Andrews on generalized third order mock theta functions \cite[p.~78, Equation (3b)]{andrews1966} which gives Ramanujan's relation between the third order mock theta functions $\phi(q)$ and $\psi(q)$ in \cite[p.~31]{lnb} as a special case. It is stated below.

\begin{corollary}\label{mock theta gen}
Let $\varphi(q)$ be defined in \eqref{varphiq}. Then,
	\begin{align}\label{mock theta gen eqn}
		\sum_{m=0}^{\infty}\frac{(-1)^mq^{m^2}}{(-bq^2;q^2)_m}=\frac{\varphi(-q)}{(-bq)_{\infty}}+(1+b)\sum_{m=1}^{\infty}\frac{(-1)^{m-1}q^{m^2}}{(-bq;q^2)_m}.
	\end{align}
	In particular, 
	\begin{align}\label{mock theta}
		\phi(q)+2\psi(q)=(-q;q^2)_{\infty}^{3}(q^2;q^2)_{\infty},
	\end{align}
	where, $\phi(q)$ and $\psi(q)$ are two of Ramanujan's third order mock theta functions defined by
	\begin{align*}
		\phi(q):=\sum_{m=0}^{\infty}\frac{q^{m^2}}{(-q^2;q^2)_m}\hspace{4mm}\text{and}\hspace{4mm}\psi(q):=\sum_{m=1}^{\infty}\frac{q^{m^2}}{(q;q^2)_m}.
	\end{align*}
\end{corollary}
There are further new corollaries that result from Theorem \ref{theorem2}. One can annihilate the theta function occurring in \eqref{theorem 2 eqn} by letting $a=-q^{-1/(2s)}$, thereby obtaining
\begin{corollary}\label{annihilation of theta}
		Let $s\in\mathbb{N}$. For $b\in\mathbb{C}$,
	\begin{align}\label{theorem 2 eqn cor1}
	\sum_{k=0}^{s-1}\left\{\sum_{m=0}^{\infty}\frac{(-1)^{sm-k}q^{(sm+k)(sm+k+1)/(2s)}}{(bq)_{m}} \sum_{n=0}^{\infty}\frac{(-b)^{n}q^{(n(n+2js-2k)-n)/(2s)}}{(q)_{n}} \right\} \notag \\ =
- (1-b) \sum_{n=1}^{\infty}(-1)^{n}q^{n(n-1)/(2s)} \sum_{\ell=0}^{n-1}\frac{b^\ell}{(q)_{\ell}}.
\end{align}
\end{corollary}
Indeed, the Jacobi triple product identity (see \eqref{jtpi} below) implies\footnote{One can also establish this evaluation by simple series manipulations.} $\sum_{n=-\infty}^{\infty}(-1)^nq^{n^2-n}=0$ upon which one can replace $q$ by $q^{1/(2s)}$. Thus, the theta function in \eqref{theorem 2 eqn} can be annihilated.

Further, if we let $b=1$ in the above corollary, we obtain
\begin{corollary}\label{rhs zero}
We have
	\begin{align}\label{theorem 2 eqn cor2}
\sum_{k=0}^{s-1}\left\{\sum_{m=0}^{\infty}\frac{(-1)^{sm-k}q^{(sm+k)(sm+k+1)/(2s)}}{(q)_{m}} \sum_{n=0}^{\infty}\frac{(-1)^{n}q^{n(n+2js-2k-1)/(2s)}}{(q)_{n}} \right\} =0.
\end{align}
\end{corollary}
A discussion on known as well as new special cases of Corollary \ref{rhs zero} is done in Section \ref{pc}. \\

If we let $s=2, b=-1$ and replace $q$ by $q^2$ in \eqref{theorem 2 eqn cor1}, we obtain
\begin{corollary}\label{cor stacks}
	We have
\begin{align}\label{involving identity from stacks}
&\sum_{m=0}^{\infty}\frac{q^{m(2m+1)}}{(-q^2;q^2)_m}\sum_{n=0}^{\infty}\frac{q^{n(n-1)/2}}{(q^2;q^2)_{n}}-\sum_{m=0}^{\infty}\frac{q^{2m^2+3m+1}}{(-q^2;q^2)_m}\sum_{n=0}^{\infty}\frac{q^{n(n+1)/2}}{(q^2;q^2)_{n}}\nonumber\\
&=-2\sum_{n=1}^{\infty}(-1)^nq^{n(n-1)/2}\sum_{\ell=0}^{n-1}\frac{(-1)^{\ell}}{(q^2;q^2)_{\ell}}.
\end{align}
\end{corollary}
The series $\sum_{n=0}^{\infty}q^{\frac{n(n+1)}{2}}/(q^2;q^2)_{n}$ is the principal object of study in \cite{andrews-stacks}, and also appears in four important identities in Ramanujan's Lost Notebook, two of which are quite difficult to prove; see for example \cite[p.~25]{lnb}, \cite{andrews-stacks}, \cite[p.~168--172]{aar2}. We note that Verma and Jain have not only given another proof of these identities of Ramanujan proved by Andrews but have also generalized them; see \cite[Equations (1.10), (1.11)]{verma-jain}.

The series $\displaystyle\sum_{n=0}^{\infty}\frac{(\pm q^{1/2})^{\pm n}q^{n(n-1)/2}}{(q^2;q^2)_{n}}$ are special cases of the series $\displaystyle\sum_{n=0}^{\infty}\frac{z^nq^{n^2/2}}{(q^2;q^2)_{n}}$ which is a $q$-analogue of Airy function and which arises as a hypergeometric solution of a $q$-Painlev\'{e} equation \cite[2512, Equation (3.37)]{kajiwara}.

\section{Preliminaries}\label{prelim}

For $|z|<1$, Euler derived the identity \cite[p.~9, Equation (1.3.6)]{ntsr}
\begin{align}\label{euler}
\sum_{\ell=0}^{\infty}\frac{z^{\ell}}{(q)_{\ell}}=\frac{1}{(z)_\infty}.
\end{align}
Another result of Euler we will need in the sequel is \cite[p.~9, Equation (1.3.7)]{ntsr}
\begin{align}\label{euler2}
\sum_{n=0}^{\infty}\frac{(-z)^nq^{n(n-1)/2}}{(q)_n}=(z)_\infty\hspace{7mm}(|z|<\infty).	
\end{align}
Ramanujan \cite[p.~359]{lnb} (see also \cite[Equation (1.9)]{mcintosh}, \cite[p.~269]{ramanujanIV}) obtained the following asymptotic formula for the series $\sum_{n=0}^{\infty}a^nq^{bn^2+cn}/(q)_n$.

\begin{theorem}\label{ramanujan asymptotic dilog}
Let $a, b, c$, and $q$ be real numbers such that $a>0, b>0$ and $|q|<1$. Let $z$ denote the positive root of $az^{2b}+z-1=0$. Then as $q\to1^{-}$,
\begin{align*}
\sum_{n=0}^{\infty}\frac{a^nq^{bn^2+cn}}{(q)_n}\sim\frac{z^c}{\sqrt{z+2b (1-z)}}\exp{\left(-\frac{1}{\log(q)}\left(\textup{Li}_{2}(az^{2b})+b\log^{2}(z)\right)\right)},
\end{align*}
where the $\textup{Li}_2(z)$ is the dilogarithm function defined for $|z|<1$, by $\textup{Li}_{2}(z):=\sum_{n=1}^{\infty}z^n/n^2$, and for any $z\in\mathbb{C}$ by $\textup{Li}_2(z):=-\int_{0}^{z}\log(1-u)/u\, du$.
\end{theorem}

\section{Proof of the main result}\label{pmr}
We now derive a result whose first equality gives an identity of Ramanujan as a corollary in the special case $s=2$. This result is proved combinatorially.
\begin{theorem}\label{theorem1}
	Let $s\in\mathbb{N}$. For any complex numbers $a,b$ and $q$ such that $|q|<1$,
	\begin{align}\label{thm3.1}
		\sum_{n=0}^{\infty}\frac{a^{n}b^{n}q^{n^2}}{(aq)_{n}(bq)_{n}} =
		\sum_{k=0}^{s-1}\sum_{n=0}^{\infty}\frac{a^{sn+k}b^{n+j}q^{(sn+k)(n+j)}}{(aq)_{n}(bq)_{sn+k}} =
		1+b\sum_{n=1}^{\infty}\frac{a^{n}q^{n}}{(bq)_{n}},
	\end{align}
	where $j$ is defined in \eqref{j}.
\end{theorem}
\begin{proof}
The equality of the extreme sides of \eqref{thm3.1} is a result of Andrews who proved it in \cite[p.~24, Corollary 2.2]{yesto}. The proof of the first equality is an adaptation of the proof of the identity given in \cite[p.~21, Theorem 2.1]{yesto} which is the special case $s=2$ of our first equality in \eqref{thm3.1}.

We show that coefficient of $b^ma^rq^n$ on both sides of the first equality is the number of partitions of $n$ into $m$ parts with the largest part $r$. This is easy to see for the sum $\displaystyle\sum_{n=0}^{\infty}\frac{a^{n}b^{n}q^{n^2}}{(aq)_{n}(bq)_{n}}$ by considering a partition with the Durfee square of size $n\times n$ and by noting that $1/(bq)_n$ generates the partition below the Durfee square whereas $1/(aq)_n$ does the same for the partition to the right of the Durfee square.

We now show that the sum $\displaystyle\sum_{k=0}^{s-1}\sum_{n=0}^{\infty}\frac{a^{sn+k}b^{n+j}q^{(sn+k)(n+j)}}{(aq)_{n}(bq)_{sn+k}}$, where $j$ is defined in \eqref{j}, generates the same partitions alluded to at the beginning of the previous paragraph.

First, consider a Durfee rectangle of $n$ rows and $sn$ columns of nodes. We attach a partition $\pi_1$ to the right of these $sn^2$ nodes with at most $n$ nodes in each column. Also, we attach a partition $\pi_2$ below the Durfee rectangle with at most $sn$ nodes in each row. Thus, the total number of parts in the entire partition is 
\begin{equation*}
n+\text{the number of parts of}\hspace{1mm}\pi_2
\end{equation*}
 whereas the largest part of the partition is 
 \begin{equation*}
 sn+\text{the number of columns of}\hspace{1mm}\pi_1.
\end{equation*}
   Thus, the generating function for all such partitions is 
   \begin{equation}\label{k=0} \frac{a^{sn}b^{n}q^{sn^2}}{(aq)_{n}(bq)_{sn}}.
\end{equation}
We also need to consider partitions with Durfee rectangles with  $n+1$ rows and $sn+k$ columns of nodes, for all $k$ such that $1\leq k\leq s-1$ to get all partitions claimed to be enumerated by the coefficient of $b^ma^rq^n$. Note that in this case, the partition $\pi_1$ will have at most $n$ nodes in each column while the partition $\pi_2$ will have at most $sn+k$ nodes in each row. Thus the generating function associated with the concerned Durfee rectangles is
\begin{equation}\label{other k}
 \sum_{k=1}^{s-1}\frac{a^{sn+k}b^{n+1}q^{(sn+k)(n+1)}}{(aq)_{n}(bq)_{sn+k}}.
 \end{equation}

If we now sum \eqref{k=0} as well as \eqref{other k} over $n$ from $0$ to $\infty$ and add the resulting sides, we obtain 
\begin{align*}
\sum_{n=0}^{\infty}\sum_{k=0}^{s-1}\frac{a^{sn+k}b^{n+j}q^{(sn+k)(n+j)}}{(aq)_{n}(bq)_{sn+k}},
\end{align*}
which is the required generating function. This proves the first equality in \eqref{thm3.1}.
\end{proof}

We will be requiring the following two corollaries of Theorem \ref{theorem1} in the sequel.
\begin{corollary}\label{corollary1}
	Let $s, n\in\mathbb{N}$. For any complex numbers $b$ and $q$ such that $|q|<1$,
	\begin{align}\label{cor3.2}
		\sum_{k=0}^{s-1}\sum_{m=0}^{\infty}\frac{b^{sm+k}q^{(sm+k+n)(m+j)}}{(bq)_{m}(q^{n+1})_{sm+k}} =
		1+\sum_{m=1}^{\infty}\frac{b^{m}q^{m+n}}{(q^{n+1})_{m}},
	\end{align}
	where $j$ is defined in \eqref{j}.
\end{corollary}
\begin{proof}
	From the second equality of Theorem \ref{theorem1}, we have
	\begin{align*}
	\sum_{k=0}^{s-1}\sum_{m=0}^{\infty}\frac{a^{sm+k}b^{m+j}q^{(sm+k)(m+j)}}{(aq)_{m}(bq)_{sm+k}} =
	1+b\sum_{m=1}^{\infty}\frac{a^{m}q^{m}}{(bq)_{m}}.
	\end{align*}
	Now let $b=q^n$ and then replace $a$ by $b$ to arrive at \eqref{cor3.2}.
\end{proof}
\begin{corollary}\label{corollary2}
	For any complex numbers $b$ and $q$ such that $|q|<1$ and $m\in\mathbb{N}$,
	\begin{align*}
		\sum_{k=0}^{s-1}\sum_{n=0}^{\infty}\frac{b^{sn+k}q^{(sn+k)(m+n+j)}}{(bq^{m+1})_{n}(q)_{sn+k}} =
		\frac{1}{(bq^{m+1})_{\infty}},
	\end{align*}
	where $j$ is as in \eqref{j}.
\end{corollary}
\begin{proof}
Let $b=1$ in the second equality of Theorem \ref{theorem1} and then let $a=bq^{m}$ to obtain
\begin{align*}
	\sum_{k=0}^{s-1}\sum_{n=0}^{\infty}\frac{b^{sn+k}q^{(sn+k)(m+n+j)}}{(bq^{m+1})_{n}(q)_{sn+k}} =\sum_{n=0}^{\infty}\frac{(bq^{m+1})^n}{(q)_n}.
	\end{align*}
	The result now follows for $|bq^{m+1}|<1$ by using \eqref{euler} to represent the sum on the right-hand side as an infinite product. The result now follows for all $b$ by analytic continuation.
\end{proof}
\begin{theorem}\label{theorem2-1}
	We have
		\begin{gather} 
		\sum_{k=0}^{s-1}\left\{\sum_{m=0}^{\infty}\frac{a^{-sm-k}q^{(sm+k)^2/(2s)}}{(bq)_{m}} \sum_{n=sm+k+1}^{\infty}\frac{a^{n}b^{n}q^{n(n+2js-2k)/(2s)}}{(q)_{n}} \right\} \notag \\ =
		\frac{1}{(bq)_{\infty}} \sum_{n=1}^{\infty}a^{n}q^{n^2/(2s)} - (1-b) \sum_{n=1}^{\infty}a^{n}q^{n^2/(2s)} \sum_{\ell=0}^{n-1}\frac{b^\ell}{(q)_{\ell}}, \label{firstPart}
	\end{gather}
where $j$ is defined in \eqref{j}.
\end{theorem}

\begin{proof}
Let $S_1$ denote the left-hand side of \eqref{firstPart}. Replacing $n$ by $n+sm+k$ in the sum over $n$ in $S_1$, interchanging the order of summation, and invoking  Corollary \ref{corollary1}, we see that
\begin{align*}
S_1&=\sum_{k=0}^{s-1}\left\{\sum_{m=0}^{\infty}\frac{a^{-sm-k}q^{(sm+k)^2/(2s)}}{(bq)_{m}} \sum_{n=1}^{\infty}\frac{a^{n+sm+k}b^{n+sm+k}q^{(n+sm+k)(n+sm+2js-k)/(2s)}}{(q)_{n+sm+k}} \right\}\nonumber\\
&=\sum_{n=1}^{\infty} \frac{a^{n}b^{n}q^{n^2/(2s)}}{(q)_n} \left\{\sum_{k=0}^{s-1}\sum_{m=0}^{\infty}\frac{b^{sm+k}q^{(sm+k+n)(m+j)}}{(bq)_{m}(q^{n+1})_{sm+k}} \right\}\nonumber\\
&=\sum_{n=1}^{\infty} \frac{a^{n}b^{n}q^{n^2/(2s)}}{(q)_n} \left\{1+\sum_{m=1}^{\infty}\frac{b^{m}q^{m+n}}{(q^{n+1})_{m}} \right\}\nonumber\\
&=\sum_{n=1}^{\infty} a^{n}q^{n^2/(2s)} \left\{\frac{b^n}{(q)_n}+\sum_{\ell=n+1}^{\infty}\frac{b^{\ell}q^{\ell}}{(q)_{\ell}} \right\},
\end{align*}
where, in the last step, we let $m=\ell-n$. Next, using \eqref{euler} with $z=bq$, we see that
\begin{align*}
S_1&=\sum_{n=1}^{\infty} a^{n}q^{n^2/(2s)} \left\{\frac{b^n}{(q)_n}+\frac{1}{(bq)_{\infty}}-\sum_{\ell=0}^{n}\frac{b^{\ell}q^{\ell}}{(q)_{\ell}} \right\}\nonumber\\
&=\sum_{n=1}^{\infty} a^{n}q^{n^2/(2s)} \left\{\frac{b^n}{(q)_n}+\frac{1}{(bq)_{\infty}}-\sum_{\ell=0}^{n}\frac{b^{\ell}}{(q)_{\ell}}+\sum_{\ell=0}^{n}\frac{b^{\ell}}{(q)_{\ell-1}} \right\}\nonumber\\
&=\sum_{n=1}^{\infty} a^{n}q^{n^2/(2s)} \left\{\frac{1}{(bq)_{\infty}}-\sum_{\ell=0}^{n-1}\frac{b^{\ell}}{(q)_{\ell}}+\sum_{\ell=1}^{n}\frac{b^{\ell}}{(q)_{\ell-1}} \right\},
\end{align*}
where, in the penultimate step, we used the elementary identity $q^\ell/(q)_{\ell}=1/(q)_\ell-1/(q)_{\ell-1}$, and in the ultimate step, we used the fact $1/(q)_{-1}=0$. Now combine the two finite sums over $\ell$ to arrive at \eqref{firstPart} after simplification.
\end{proof}

\begin{theorem}\label{theorem2-2}
	We have
	\begin{gather}
		\sum_{k=0}^{s-1}\left\{\sum_{m=0}^{\infty}\frac{a^{-sm-k}q^{(sm+k)^2/(2s)}}{(bq)_{m}} \sum_{n=0}^{sm+k}\frac{a^{n}b^{n}q^{n(n+2js-2k)/(2s)}}{(q)_{n}} \right\} =
		\frac{1}{(bq)_{\infty}} \sum_{n=-\infty}^{0}a^{n}q^{n^2/(2s)}, \label{secondPart}
	\end{gather}
where $j$ is defined in \eqref{j}.
\end{theorem}

\begin{proof}
Let $S_2$ denote the left-hand side of \eqref{secondPart}.
Fix $k$ such that $0\leq k\leq s-1$, and consider the finite sum over $n$ on the left-hand side of \eqref{secondPart}. Split this sum according to the residue classes of $n$ modulo $s$, that is, replace $n$ by $sn+\ell$, where $0\leq\ell\leq s-1$, to get
\begin{gather*}
	\sum_{n=0}^{sm+k}\frac{a^{n}b^{n}q^{n(n+2js-2k)/(2s)}}{(q)_{n}} =
	\sum_{\ell=0}^{s-1}\sum_{n=0}^{m-i}\frac{a^{sn+\ell}b^{sn+\ell}q^{(sn+\ell)(sn+\ell+2js-2k)/(2s)}}{(q)_{sn+\ell}},
\end{gather*}
where
\begin{align}\label{i}
	i= \left\{
	\begin{array}{ll}
		0,  & \mbox{if } 0\le \ell \le k, \\
		1, & \mbox{if } \ell>k. 
	\end{array}
	\right. 
\end{align}
Thus,
\begin{align*}
	S_2=\sum_{k=0}^{s-1}\sum_{m=0}^{\infty}\frac{a^{-sm-k}q^{(sm+k)^2/(2s)}}{(bq)_{m}} \left( \sum_{l=0}^{s-1}\sum_{n=0}^{m-i}\frac{a^{sn+l}b^{sn+l}q^{(sn+l)(sn+l+2js-2k)/(2s)}}{(q)_{sn+l}} \right). 
\end{align*}
We now change the order of summation in the above sum in the following way. For $0\leq t\leq s-1$, define $G_t$ to be that part of $S_2$ for which $\ell\equiv(k+t)\pmod{s}$. In view of \eqref{i} and the fact that $0\leq k,\ell\leq s-1$ and $0<t\leq s-1$, we see that
\begin{align*}
	G_0 &:= \sum_{k=0}^{s-1}\left\{\sum_{m=0}^{\infty}\frac{a^{-sm-k}q^{(sm+k)^2/(2s)}}{(bq)_{m}} \sum_{n=0}^{m}\frac{a^{sn+k}b^{sn+k}q^{(sn+k)(sn+2js-k)/(2s)}}{(q)_{sn+k}} \right\},\\
 	G_t &:= \sum_{k=0}^{s-t-1}\left\{\sum_{m=0}^{\infty}\frac{a^{-sm-k}q^{(sm+k)^2/(2s)}}{(bq)_{m}} \sum_{n=0}^{m-1}\frac{a^{sn+k+t}b^{sn+k+t}q^{(sn+k+t)(sn+2js+t-k)/(2s)}}{(q)_{sn+k+t}} \right\}  \notag\\
 	&\quad+\sum_{k=s-t}^{s-1}\left\{\sum_{m=0}^{\infty}\frac{a^{-sm-k}q^{(sm+k)^2/(2s)}}{(bq)_{m}} \sum_{n=0}^{m}\frac{a^{sn+k+t-s}b^{sn+k+t-s}q^{(sn+k+t-s)(sn+2js+t-s-k)/(2s)}}{(q)_{sn+k+t-s}} \right\}.
 \end{align*}
For the case $s=6$, the distribution of the terms of $S_2$ into each of $G_t$ with respect to the different values of $k$ and $\ell$ is illustrated in the diagram below.
\begin{figure}[hbt!]
	\begin{center}
		\includegraphics[width=0.40\textwidth]{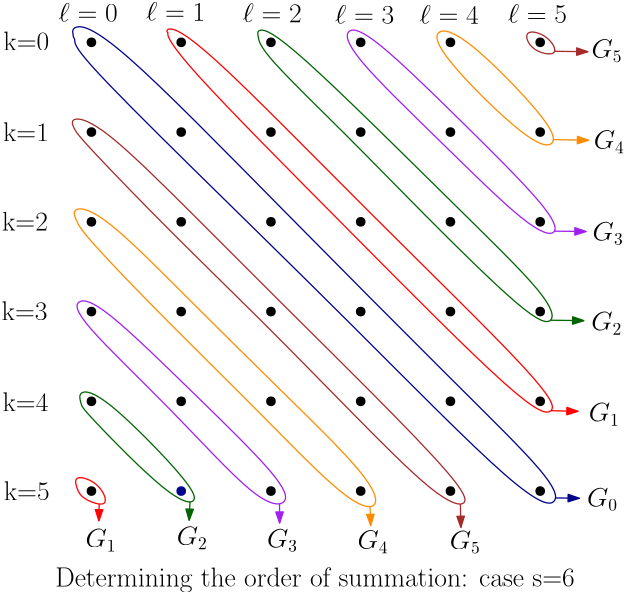}
	\end{center}
\end{figure}

Therefore,
\begin{align}
S_2=\sum_{t=0}^{s-1} G_t. \label{gtSum}
\end{align}
The next task is to evaluate $G_t$ for $0\leq t\leq s-1$. We begin with $G_0$. 

Interchanging the order of summation of the sums over $m$ and $n$, and then replacing $m$ by $m+n$ in the second step, we find that
\begin{align}
G_0 &= \sum_{k=0}^{s-1}\left\{\sum_{n=0}^{\infty}\sum_{m=n}^{\infty}\frac{a^{sn-sm}b^{sn+k}q^{((sm+k)^{2}+(sn+k)(sn+2js-k))/(2s)}}{(bq)_{m}(q)_{sn+k}} \right\}\nonumber\\
&=	\sum_{m=0}^{\infty} \frac{a^{-sm}q^{(sm)^2/(2s)}}{(bq)_m} \left\{\sum_{k=0}^{s-1}\sum_{n=0}^{\infty}\frac{b^{sn+k}q^{(sn+k)(n+m+j)}}{(bq^{m+1})_{n}(q)_{sn+k}} \right\}\nonumber\\
&=\sum_{m=0}^{\infty} \frac{a^{-sm}q^{(sm)^2/(2s)}}{(bq)_m} \frac{1}{(bq^{m+1})_{\infty}}  \notag\\
&=\frac{1}{(bq)_{\infty}} \sum_{m=0}^{\infty}a^{-sm}q^{(sm)^2/(2s)}, \label{G0value}
\end{align}
where, in the penultimate step, we used Corollary \ref{corollary2}.

To obtain a simplified expression for $G_t, 0<t\leq s-1$, interchange the order of summation in each of the double sums over $m$ and $n$ to obtain
\begin{align}
	G_t &= \sum_{k=0}^{s-t-1}\left\{\sum_{n=0}^{\infty} \sum_{m=n+1}^{\infty}\frac{a^{sn-sm+t}b^{sn+k+t}q^{((sm+k)^{2}+(sn+k+t)(sn+2js+t-k))/(2s)}}{(bq)_{m}(q)_{sn+k+t}} \right\}  \notag\\
&\quad+\sum_{k=s-t}^{s-1}\left\{\sum_{n=0}^{\infty} \sum_{m=n}^{\infty}\frac{a^{sn-sm+t-s}b^{sn+k+t-s}q^{((sm+k)^{2}+(sn+k+t-s)(sn+2js+t-s-k))/(2s)}}{(bq)_{m}(q)_{sn+k+t-s}} \right\}\notag\\
 &= \sum_{k=0}^{s-t-1}\left\{ \sum_{n=0}^{\infty} \sum_{m=0}^{\infty}\frac{a^{-sm-s+t}b^{sn+k+t}q^{((sm+sn+s+k)^{2}+(sn+k+t)(sn+2js+t-k))/(2s)}}{(bq)_{m+n+1}(q)_{sn+k+t}} \right\}  \notag\\
&\quad+\sum_{k=s-t}^{s-1}\left\{ \sum_{n=0}^{\infty} \sum_{m=0}^{\infty}\frac{a^{-sm-s+t}b^{sn+k+t-s}q^{((sm+sn+k)^{2}+(sn+k+t-s)(sn+2js+t-s-k))/(2s)}}{(bq)_{m+n}(q)_{sn+k+t-s}} \right\} \notag\\
&=\sum_{m=0}^{\infty} \frac{a^{-sm-s+t}q^{(sm+s-t)^2/(2s)}}{(bq)_m} \Bigg\{ \sum_{k=0}^{s-t-1} \sum_{n=0}^{\infty}\frac{b^{sn+k+t}q^{(m+n+j+1)(sn+k+t)}}{(bq^{m+1})_{n+1}(q)_{sn+k+t}} \notag\\ &\quad+\sum_{k=s-t}^{s-1} \sum_{n=0}^{\infty} \frac{b^{sn+k+t-s}q^{(m+n+j)(sn+k+t-s)}}{(bq^{m+1})_{n}(q)_{sn+k+t-s}}  \Bigg\}. \label{gteval}
\end{align}
Now let $0<t<s-1$. We now separate $k=0$ term from the first double sum in the curly braces and $k=s-t$ term from the second, use the elementary fact $\frac{1}{1-bq^{m+n+1}}=1+\frac{bq^{m+n+1}}{1-bq^{m+n+1}}$ in the resulting first double sum, and recall the definition of $j$ from \eqref{j} to get
\begin{align}
	G_t &= \sum_{m=0}^{\infty} \frac{a^{-sm-s+t}q^{(sm+s-t)^2/(2s)}}{(bq)_m} \Bigg\{ \sum_{n=0}^{\infty}\frac{b^{sn+t}q^{(m+n+1)(sn+t)}}{(bq^{m+1})_{n+1}(q)_{sn+t}}  +  \sum_{k=1}^{s-t-1} \sum_{n=0}^{\infty}\frac{b^{sn+k+t}q^{(m+n+2)(sn+k+t)}}{(bq^{m+1})_{n}(q)_{sn+k+t}} \nonumber\\
	&\quad+ \sum_{k=1}^{s-t-1} \sum_{n=0}^{\infty}\frac{b^{sn+k+t+1}q^{(m+n+2)(sn+k+t)+m+n+1}}{(bq^{m+1})_{n+1}(q)_{sn+k+t}} + \sum_{n=0}^{\infty} \frac{b^{sn}q^{(m+n+1)(sn)}}{(bq^{m+1})_{n}(q)_{sn}}\nonumber\\
	&\quad+ \sum_{k=s-t+1}^{s-1} \sum_{n=0}^{\infty} \frac{b^{sn+k+t-s}q^{(m+n+1)(sn+k+t-s)}}{(bq^{m+1})_{n}(q)_{sn+k+t-s}}  \Bigg\}. \label{gtAB}
\end{align}
Now define $S_3$ to be the sum of the second and third double sums in the curly braces of \eqref{gtAB}, that is, let
\begin{align*}
S_3:=\sum_{k=1}^{s-t-1} \sum_{n=0}^{\infty}\frac{b^{sn+k+t}q^{(m+n+2)(sn+k+t)}}{(bq^{m+1})_{n}(q)_{sn+k+t}} + \sum_{k=1}^{s-t-1} \sum_{n=0}^{\infty}\frac{b^{sn+k+t+1}q^{(m+n+2)(sn+k+t)+m+n+1}}{(bq^{m+1})_{n+1}(q)_{sn+k+t}} .
\end{align*}
Using $q^{sn+k+t}=1-(1-q^{sn+k+t})$, we see that $S_3$ can be put in the form
\begin{align*}
S_3&=\sum_{k=1}^{s-t-1} \sum_{n=0}^{\infty}\frac{b^{sn+k+t}q^{(m+n+1)(sn+k+t)}}{(bq^{m+1})_{n}(q)_{sn+k+t}} - \sum_{k=1}^{s-t-1} \sum_{n=0}^{\infty}\frac{b^{sn+k+t}q^{(m+n+1)(sn+k+t)}}{(bq^{m+1})_{n}(q)_{sn+k+t-1}} \notag\\
&\quad+ \sum_{k=1}^{s-t-1} \sum_{n=0}^{\infty}\frac{b^{sn+k+t+1}q^{(m+n+1)(sn+k+t+1)}}{(bq^{m+1})_{n+1}(q)_{sn+k+t}} - \sum_{k=1}^{s-t-1} \sum_{n=0}^{\infty}\frac{b^{sn+k+t+1}q^{(m+n+1)(sn+k+t+1)}}{(bq^{m+1})_{n+1}(q)_{sn+k+t-1}}.
\end{align*}
Next, combine the second and the fourth double sum on the right-hand side of the above equation using $1+\frac{bq^{m+n+1}}{1-bq^{m+n+1}}=\frac{1}{1-bq^{m+n+1}}$ so as to have
\begin{align*}
S_3&=\sum_{k=1}^{s-t-1} \sum_{n=0}^{\infty}\frac{b^{sn+k+t}q^{(m+n+1)(sn+k+t)}}{(bq^{m+1})_{n}(q)_{sn+k+t}} - \sum_{k=1}^{s-t-1} \sum_{n=0}^{\infty}\frac{b^{sn+k+t}q^{(m+n+1)(sn+k+t)}}{(bq^{m+1})_{n+1}(q)_{sn+k+t-1}} \notag\\
&\quad+ \sum_{k=1}^{s-t-1} \sum_{n=0}^{\infty}\frac{b^{sn+k+t+1}q^{(m+n+1)(sn+k+t+1)}}{(bq^{m+1})_{n+1}(q)_{sn+k+t}}.
\end{align*}
Now we separate the $k=1$ term of the second double sum, then replace $k$ by $k+1$ in the resulting double sum. Similarly, we separate the $k=s-t-1$ term from the third double sum. This has the effect of the resulting two double sums completely getting canceled thereby leading to
\begin{align*}
S_3&=
\sum_{k=1}^{s-t-1} \sum_{n=0}^{\infty}\frac{b^{sn+k+t}q^{(m+n+1)(sn+k+t)}}{(bq^{m+1})_{n}(q)_{sn+k+t}} - \sum_{n=0}^{\infty}\frac{b^{sn+t+1}q^{(m+n+1)(sn+t+1)}}{(bq^{m+1})_{n+1}(q)_{sn+t}} \notag\\
&\quad+ \sum_{n=0}^{\infty}\frac{b^{sn+s}q^{(m+n+1)(sn+s)}}{(bq^{m+1})_{n+1}(q)_{sn+s-1}}.
\end{align*}  
We now substitute the above representation of $S_3$ back in \eqref{gtAB} to get
\begin{align*}
G_t&= \sum_{m=0}^{\infty} \frac{a^{-sm-s+t}q^{(sm+s-t)^2/(2s)}}{(bq)_m} \Bigg\{ \sum_{n=0}^{\infty}\frac{b^{sn+t}q^{(m+n+1)(sn+t)}}{(bq^{m+1})_{n+1}(q)_{sn+t}} 
+ \Bigg(\sum_{k=1}^{s-t-1} \sum_{n=0}^{\infty}\frac{b^{sn+k+t}q^{(m+n+1)(sn+k+t)}}{(bq^{m+1})_{n}(q)_{sn+k+t}} \notag\\
&\quad- \sum_{n=0}^{\infty}\frac{b^{sn+t+1}q^{(m+n+1)(sn+t+1)}}{(bq^{m+1})_{n+1}(q)_{sn+t}} + \sum_{n=0}^{\infty}\frac{b^{sn+s}q^{(m+n+1)(sn+s)}}{(bq^{m+1})_{n+1}(q)_{sn+s-1}}\Bigg)+\sum_{n=0}^{\infty} \frac{b^{sn}q^{(m+n+1)(sn)}}{(bq^{m+1})_{n}(q)_{sn}}\notag\\
&\quad  + \sum_{k=s-t+1}^{s-1} \sum_{n=0}^{\infty} \frac{b^{sn+k+t-s}q^{(m+n+1)(sn+k+t-s)}}{(bq^{m+1})_{n}(q)_{sn+k+t-s}}  \Bigg\}.
\end{align*}
Now combine the first and the third sum in the curly braces, and replace $n$ by $n-1$ in the fourth to have
\begin{align*}
G_t &= \sum_{m=0}^{\infty} \frac{a^{-sm-s+t}q^{(sm+s-t)^2/(2s)}}{(bq)_m} \Bigg\{ \sum_{n=0}^{\infty}\frac{b^{sn+t}q^{(m+n+1)(sn+t)}}{(bq^{m+1})_{n}(q)_{sn+t}} +  \sum_{k=1}^{s-t-1} \sum_{n=0}^{\infty}\frac{b^{sn+k+t}q^{(m+n+1)(sn+k+t)}}{(bq^{m+1})_{n}(q)_{sn+k+t}}\notag\\
&\quad+ \sum_{n=1}^{\infty}\frac{b^{sn}q^{(m+n)(sn)}}{(bq^{m+1})_{n}(q)_{sn-1}} + \sum_{n=0}^{\infty} \frac{b^{sn}q^{(m+n+1)(sn)}}{(bq^{m+1})_{n}(q)_{sn}} + \sum_{k=s-t+1}^{s-1} \sum_{n=0}^{\infty} \frac{b^{sn+k+t-s}q^{(m+n+1)(sn+k+t-s)}}{(bq^{m+1})_{n}(q)_{sn+k+t-s}}  \Bigg\}.
\end{align*}
The first expression in the curly braces can be inducted as the $k=0$ term of the second. Also, since $1/(q)_{-1}=0$, the third sum can be made to start from $n=0$, and can then be combined with the fourth sum using $1+\frac{q^{sn}}{1-q^{sn}}=\frac{1}{1-q^{sn}}$. This leads to
\begin{align*}
	G_t &= \sum_{m=0}^{\infty} \frac{a^{-sm-s+t}q^{(sm+s-t)^2/(2s)}}{(bq)_m} \Bigg\{\sum_{k=0}^{s-t-1} \sum_{n=0}^{\infty}\frac{b^{sn+k+t}q^{(m+n+1)(sn+k+t)}}{(bq^{m+1})_{n}(q)_{sn+k+t}} + \sum_{n=0}^{\infty} \frac{b^{sn}q^{(m+n)(sn)}}{(bq^{m+1})_{n}(q)_{sn}} \notag\\
	&\quad+ \sum_{k=s-t+1}^{s-1} \sum_{n=0}^{\infty} \frac{b^{sn+k+t-s}q^{(m+n+1)(sn+k+t-s)}}{(bq^{m+1})_{n}(q)_{sn+k+t-s}}  \Bigg\}\nonumber\\
	&= \sum_{m=0}^{\infty} \frac{a^{-sm-s+t}q^{(sm+s-t)^2/(2s)}}{(bq)_m} \Bigg\{ \sum_{k=t}^{s-1} \sum_{n=0}^{\infty}\frac{b^{sn+k}q^{(m+n+1)(sn+k)}}{(bq^{m+1})_{n}(q)_{sn+k}} +\sum_{n=0}^{\infty} \frac{b^{sn}q^{(m+n)(sn)}}{(bq^{m+1})_{n}(q)_{sn}} \notag\\
	&\quad+ \sum_{k=1}^{t-1} \sum_{n=0}^{\infty} \frac{b^{sn+k}q^{(m+n+1)(sn+k)}}{(bq^{m+1})_{n}(q)_{sn+k}}  \Bigg\},
\end{align*}
where, in the last step, we replaced $k$ by $k-t$ in the first sum, and $k$ by $k+s-t$ in the third.

Now observe an interesting thing -  the expressions in the curly braces now combine to give us the single expression, namely,
\begin{align*}
\sum_{k=0}^{s-1}\sum_{n=0}^{\infty}\frac{b^{sn+k}q^{(sn+k)(m+n+j)}}{(bq^{m+1})_{n}(q)_{sn+k}},
\end{align*}
where $j$ is defined in \eqref{j}. Using Corollary \eqref{corollary2}, we see that
\begin{align}
	G_t &= \sum_{m=0}^{\infty} \frac{a^{-sm-s+t}q^{(sm+s-t)^2/(2s)}}{(bq)_m}\cdot \frac{1}{(bq^{m+1})_{\infty}}  \notag\\
	&= \frac{1}{(bq)_{\infty}} \sum_{m=0}^{\infty}a^{-sm-s+t}q^{(sm+s-t)^2/(2s)}. \label{Gtvalue}
\end{align}
The only thing remaining at this stage is to obtain a simplified expression for $G_{s-1}$. To that end, observe that letting $t=s-1$ in \eqref{gteval},
	\begin{align*}
	G_{s-1} = \sum_{m=0}^{\infty} \frac{a^{-sm-1}q^{(sm+1)^2/(2s)}}{(bq)_m} \left\{ \sum_{n=0}^{\infty}\frac{b^{sn+s-1}q^{(m+n+1)(sn+s-1)}}{(bq^{m+1})_{n+1}(q)_{sn+s-1}} +  \sum_{k=1}^{s-1} \sum_{n=0}^{\infty} \frac{b^{sn+k-1}q^{(m+n+1)(sn+k-1)}}{(bq^{m+1})_{n}(q)_{sn+k-1}}  \right\}.
\end{align*}
We now split the first expression in the curly braces using  $\frac{1}{1-bq^{m+n+1}}=1+\frac{bq^{m+n+1}}{1-bq^{m+n+1}}$, then separate out $k=1$ term of the double sum, and then replace $k$ by $k+1$ so as to get
	\begin{align*}
	G_{s-1} &
	= \sum_{m=0}^{\infty} \frac{a^{-sm-1}q^{(sm+1)^2/(2s)}}{(bq)_m} \left\{ \sum_{n=0}^{\infty}\frac{b^{sn+s-1}q^{(m+n+1)(sn+s-1)}}{(bq^{m+1})_{n}(q)_{sn+s-1}} + \sum_{n=0}^{\infty}\frac{b^{sn+s}q^{(m+n+1)(sn+s)}}{(bq^{m+1})_{n+1}(q)_{sn+s-1}} \right. \notag\\ &\left. \quad+ \sum_{n=0}^{\infty} \frac{b^{sn}q^{(m+n+1)(sn)}}{(bq^{m+1})_{n}(q)_{sn}} +\sum_{k=1}^{s-2} \sum_{n=0}^{\infty} \frac{b^{sn+k}q^{(m+n+1)(sn+k)}}{(bq^{m+1})_{n}(q)_{sn+k}}  \right\}.
\end{align*}
The first expression in the curly braces can now be inducted as the $k=s-1$ term of the fourth expression. Along with replacing $n$ by $n-1$ in the second expression, this gives\footnote{Note that the second expression in the curly braces of \eqref{3.28} can be made to begin at $n=0$ because $1/(q)_{-1}=0$.}
\begin{align}\label{3.28}
	G_{s-1} =& \sum_{m=0}^{\infty} \frac{a^{-sm-1}q^{(sm+1)^2/(2s)}}{(bq)_m} \left\{ \sum_{k=1}^{s-1} \sum_{n=0}^{\infty} \frac{b^{sn+k}q^{(m+n+1)(sn+k)}}{(bq^{m+1})_{n}(q)_{sn+k}} + \sum_{n=0}^{\infty}\frac{b^{sn}q^{(m+n)(sn)}}{(bq^{m+1})_{n}(q)_{sn-1}} \right. \notag\\ &\left. + \sum_{n=0}^{\infty} \frac{b^{sn}q^{(m+n+1)(sn)}}{(bq^{m+1})_{n}(q)_{sn}}  \right\}.
\end{align}
Now combine the last two sums over $n$ in the curly braces using $1+\frac{q^{sn}}{1-q^{sn}}=\frac{1}{1-q^{sn}}$, and then use Corollary \eqref{corollary2} to get
	\begin{align}
	G_{s-1} &= \sum_{m=0}^{\infty} \frac{a^{-sm-1}q^{(sm+1)^2/(2s)}}{(bq)_m} \cdot \frac{1}{(bq^{m+1})_{\infty}}  \notag\\
	&= \frac{1}{(bq)_{\infty}} \sum_{m=0}^{\infty}a^{-sm-1}q^{(sm+1)^2/(2s)}. \label{Gs1value}
\end{align}
Finally, substituting the expressions for $G_0$, $G_t$, where $0<t<s-1$, and $G_{s-1}$ given in \eqref{G0value},\eqref{Gtvalue} and \eqref{Gs1value} respectively into \eqref{gtSum} to get
\begin{align*}
	S_2&=	\frac{1}{(bq)_{\infty}} \left\{ \sum_{m=0}^{\infty}a^{-sm}q^{(sm)^2/(2s)} + \sum_{t=1}^{s-1}\sum_{m=0}^{\infty}a^{-sm-s+t}q^{(sm+s-t)^2/(2s)} \right\}\notag\\
	&=		\frac{1}{(bq)_{\infty}} \left\{ \sum_{m=0}^{\infty}a^{-sm}q^{(sm)^2/(2s)} + \sum_{t=1}^{s-1}\sum_{m=0}^{\infty}a^{-sm-t}q^{(sm+t)^2/(2s)} \right\} \notag\\ 
	&=
	\frac{1}{(bq)_{\infty}} \sum_{t=0}^{s-1}\sum_{m=0}^{\infty}a^{-sm-t}q^{(sm+t)^2/(2s)}\notag\\
	&=
	\frac{1}{(bq)_{\infty}} \sum_{m=0}^{\infty}a^{-m}q^{m^2/(2s)} \notag\\
	&=\frac{1}{(bq)_{\infty}} \sum_{m=-\infty}^{0}a^{m}q^{m^2/(2s)}.
\end{align*}
This completes the proof.
\end{proof}

\begin{proof}[Theorem \textup{\ref{theorem2}}][]
The result follows from adding the corresponding sides of Theorems \ref{theorem2-1} and \eqref{theorem2-2}.	
\end{proof}

\section{Proofs of Corollaries and other deductions}\label{pc}

\begin{proof}[Corollary \textup{\ref{ramanujan27}}][]
	Letting $s=2$ in Theorem \ref{theorem2} leads to \eqref{gmr} upon simplification.
\end{proof}

\begin{proof}[Corollary \textup{\ref{gen jtpi}}][]
Let $s=1$ in Theorem \ref{theorem2} and simplify.
\end{proof}

A special case of Corollary \ref{gen jtpi} yields the Jacobi triple product identity:
\begin{corollary}
	For any complex number $a\neq0$,
	\begin{align}\label{jtpi}
		\sum_{n=-\infty}^{\infty}a^nq^{n^2}=(-aq;q^2)_{\infty}(-a^{-1}q;q^2)_{\infty}(q^2;q^2)_{\infty}.
	\end{align}
\end{corollary}
\begin{proof}
Let $b=1$ in Corollary \ref{gen jtpi}, then replace $q$ by $q^2$ so that
\begin{align}\label{intmid}
\sum_{m=0}^{\infty}\frac{a^{-m}q^{m^2}}{(q^2;q^2)_m}\sum_{n=0}^{\infty}\frac{a^nq^{{n^2}}}{(q^2;q^2)_n}=\frac{1}{(q^2;q^2)_{\infty}} \sum_{n=-\infty}^{\infty}a^{n}q^{n^2}.
\end{align} 
Now replace $q$ by $q^2$ in Euler's identity \eqref{euler2} and use the resulting identity twice, once with $z=-a^{-1}q$ and then with $z=-aq$,  to represent each infinite sum on the left-hand side of \eqref{intmid} as an infinite product.
\end{proof}

\begin{proof}[Corollary \textup{\ref{mock theta gen}}][]
We let $s=1, a=-1$ in \eqref{gmr} and replace $b$ by $-b$ and $q$ by $q^2$ to get
\begin{align}\label{mock1}
\sum_{m=0}^{\infty}\frac{(-1)^mq^{m^2}}{(-bq^2;q^2)_m}\sum_{n=0}^{\infty}\frac{b^nq^{n^2}}{(q^2;q^2)_n}=\frac{1}{(-bq^2;q^2)_{\infty}}\sum_{n=-\infty}^{\infty}(-1)^nq^{n^2}- (1+b) \sum_{n=1}^{\infty}(-1)^{n}q^{n^2} \sum_{\ell=0}^{n-1}\frac{(-b)^\ell}{(q^2;q^2)_{\ell}}.
\end{align}
But
\begin{align}\label{mock2}
\sum_{n=1}^{\infty}(-1)^{n}q^{n^2} \sum_{\ell=0}^{n-1}\frac{(-b)^\ell}{(q^2;q^2)_{\ell}}=\sum_{m=1}^{\infty}(-1)^mq^{m^2}\sum_{\ell=0}^{\infty}\frac{(bq^{2m})^{\ell}q^{\ell^2}}{(q^2;q^2)_{\ell}}.
\end{align}
Now from \eqref{mock1}, \eqref{mock2} and \eqref{euler2} applied twice (once with $q$ replaced by $q^2$ and $z$ replaced by $-bq$, and then once more, with $q$ replaced by $q^2$ and $z$ replaced by $-bq^{2m+1}$) to get
\begin{align*}
(-bq;q^2)_{\infty}\sum_{m=0}^{\infty}\frac{(-1)^mq^{m^2}}{(-bq^2;q^2)_m}=\frac{\varphi(-q)}{(-bq^2;q^2)_{\infty}}-(1+b)(-bq;q^2)_{\infty}\sum_{m=1}^{\infty}\frac{(-1)^{m}q^{m^2}}{(-bq;q^2)_m}.
\end{align*}
The result now follows from dividing both sides by $(-bq;q^2)_{\infty}$.

To obtain \eqref{mock theta} from \eqref{mock theta gen eqn}, let $b=1$ in \eqref{mock theta gen eqn} and replace $q$ by $-q$, then use \eqref{jtpi} with $a=1$ to simplify $\varphi(q)$ and then employ Euler's theorem in the form $(q;-q)_{\infty}=1/(-q;q^2)_{\infty}$.
\end{proof}

\subsection{Special cases of Corollary \ref{rhs zero}} 
If we let $s=1$ in \eqref{theorem 2 eqn cor2}, we obtain
\begin{equation*}
	\sum_{m=0}^{\infty}\frac{(-1)^{m}q^{m(m+1)/2}}{(q)_{m}} \sum_{n=0}^{\infty}\frac{(-1)^{n}q^{n(n-1)/2}}{(q)_{n}}=0,
\end{equation*}
which is trivial to prove since $\sum_{n=0}^{\infty}(-1)^{n}q^{\frac{n(n-1)}{2}}/(q)_{n}=(1)_{\infty}=0$ by Euler's theorem. Next, letting $s=2$ in \eqref{theorem 2 eqn cor2} and then replacing $q$ by $q^2$, we see that
	\begin{align*}
		\sum_{m=0}^{\infty}\frac{q^{m(2m+1)}}{(q^2;q^2)_m}\sum_{n=0}^{\infty}\frac{(-1)^{n}q^{n(n-1)/2}}{(q^2;q^2)_{n}}=\sum_{m=0}^{\infty}\frac{q^{(m+1)(2m+1)}}{(q^2;q^2)_m}\sum_{n=0}^{\infty}\frac{(-1)^{n}q^{n(n+1)/2}}{(q^2;q^2)_{n}}.
	\end{align*}
	This identity also follows trivially by applying an identity of McIntosh twice, once with $\mu=0$\footnote{The identity in the case $\mu=0$ was first obtained by Rogers; see \cite[p.~575, Equation (R1)]{andrews-auluck} and replace $z$ by $-q$ and then $q$ by $\sqrt{q}$ there. It also follows from two identities of Ramanujan, namely, the one in \cite[pp.~59-60]{watson final}, and another in \cite[p.~10]{lnb} (see also \cite[p.~18, Equation (1.4.15)]{aar2}).} and then again with $\mu=1$, and then multiplying the left-hand side of the first with the right-hand side of the other. McIntosh's identity is given by \cite[p.~135]{mcintosh}, \cite[Equation (10)]{mcintosh2}
	\begin{align*}
	\sum_{n=0}^{\infty}\frac{q^{(2n+\mu)(2n+\mu+1)/2}}{(q^2;q^2)_n}=(-q)_\infty\sum_{n=0}^{\infty}\frac{(-1)^nq^{n(n+1)/2-\mu n}}{(q^2;q^2)_n}.
	\end{align*} 
However, for $s>2$, the identities that we get from Corollary \ref{rhs zero}, are, to the best of our knowledge, new. For example, when $s=3$, we have
\begin{align*}
	&\sum_{m=0}^{\infty}\frac{(-1)^mq^{m(3m+1)/2}}{(q)_m}\sum_{n=0}^{\infty}\frac{(-1)^nq^{n (n-1)/6}}{(q)_n}-\sum_{m=0}^{\infty}\frac{(-1)^mq^{(3m+1)(3m+2)/6}}{(q)_m}\sum_{n=0}^{\infty}\frac{(-1)^nq^{n (n+3)/6}}{(q)_n}\nonumber\\
	&+\sum_{m=0}^{\infty}\frac{(-1)^mq^{(m+1)(3m+2)/2}}{(q)_m}\sum_{n=0}^{\infty}\frac{(-1)^nq^{n (n+1)/6}}{(q)_n}=0.
\end{align*}
It would be nice to see if this result can also be derived from known identities as is the case with the identities corresponding to $s=1$ and $2$.

\subsection{A congruence implied by Corollary \ref{cor stacks}}
It was mentioned at the end of the introduction that Ramanujan's Lost Notebook contains four identities involving the sum $\sum_{n=0}^{\infty}q^{\frac{n(n+1)}{2}}/(q^2;q^2)_{n}$ (see \cite[Equations (1.8)-(1.11)]{andrews-stacks}) which occurs in our identity \eqref{involving identity from stacks}. One of them is \cite[Equation (1.8)]{andrews-stacks} is
\begin{align*}
	\sum_{n=0}^{\infty}\frac{q^{n(n+1)/2}}{(q^2;q^2)_{n}}=\frac{1}{(-q;-q)_{\infty}}\left\{\sum_{n=0}^{\infty}\frac{q^{2n^2+n}}{(-q^2;q^2)_n}+2\sum_{n=1}^{\infty}q^{n^2}(-q^2;q^2)_{n-1}\right\}.
\end{align*} 
Substituting the above identity in \eqref{involving identity from stacks} and simplifying leads to
\begin{align*}
	&\sum_{m=0}^{\infty}\frac{q^{m(2m+1)}}{(-q^2;q^2)_m}\left\{\sum_{n=0}^{\infty}\frac{q^{n(n-1)/2}}{(q^2;q^2)_{n}}-\frac{1}{(-q;-q)_{\infty}}\sum_{m=0}^{\infty}\frac{q^{2m^2+3m+1}}{(-q^2;q^2)_m}\right\}\notag\\
	&\quad-\frac{2}{(-q;-q)_{\infty}}\sum_{m=0}^{\infty}\frac{q^{2m^2+3m+1}}{(-q^2;q^2)_m}\sum_{n=1}^{\infty}q^{n^2}(-q^2;q^2)_{n-1}\notag\\
	&=-2\sum_{n=1}^{\infty}(-1)^nq^{n(n-1)/2}\sum_{\ell=0}^{n-1}\frac{(-1)^{\ell}}{(q^2;q^2)_{\ell}}\notag\\
	&=2\sum_{m=0}^{\infty}\frac{q^{m(2m+1)}}{(q^2;q^2)_{2m}}-2\sum_{\ell=1}^{\infty}\sum_{m=0}^{\infty}\frac{(-1)^{\ell}q^{2(\ell+m)^2-(\ell+m)+2\ell^2}}{(q^2;q^2)_{\ell+m-1}(q^4;q^4)_{m}},
\end{align*}
where the last step follows from splitting the double sum over $n$ and $\ell$ according to the parity of $n$ and then employing Lemmas 1 and 2 from \cite{andrews-stacks} (with $\alpha=0$).

While we are unable to further simplify the above identity, the following congruence can obviously be extracted from it:
\begin{align}\label{congruence}
	&\sum_{m=0}^{\infty}\frac{q^{m(2m+1)}}{(-q^2;q^2)_m}\left\{1+\sum_{n=1}^{\infty}\frac{q^{n(n-1)/2}}{(q^2;q^2)_{n}}-\frac{1}{(-q;-q)_{\infty}}\sum_{m=0}^{\infty}\frac{q^{2m^2+3m+1}}{(-q^2;q^2)_m}\right\}\equiv0\pmod{2}.
\end{align}
The series $\sum_{n=1}^{\infty}q^{\frac{n(n-1)}{2}}/(q^2;q^2)_{n}$ occurring in the congruence is the generating function of partitions into distinct parts in which the difference between any two consecutive parts is odd. While the power series coefficients of this sum are tabulated as sequence $A179080$ on OEIS \cite{oeis}, none of the other sums occurring in \eqref{congruence} have their power series coefficients tabulated. 
\section{Asymptotic analysis}
Before Ramanujan proved the Rogers-Ramanujan identities
\begin{align}
	G(q)&=\frac{1}{(q;q^5)_{\infty}(q^4;q^5)_{\infty}},\label{rr1}\\ H(q)&=\frac{1}{(q^2;q^5)_{\infty}(q^3;q^5)_{\infty}},\notag
\end{align}
where $G(q)$ and $H(q)$ are defined in \eqref{rrf defn}, his belief in them stemmed from several pieces of evidence he found for their existence \cite[p.~358, 360--361]{lnb}\footnote{It is to be noted that this is \emph{not} a part of the Lost Notebook, the $139$ handwritten pages found by George E. Andrews, but from the material titled `Other unpublished material including portions of manuscripts'.}. One of them was that both sides of \eqref{rr1} are asymptotically equal to $\exp{\left(\frac{\pi^2}{15(1-q)}\right)}$ as $q\to1^{-}$. 

On page $359$ of \cite{lnb}, Ramanujan obtained an asymptotic formula for the series $	\sum_{n=0}^{\infty}a^{n}q^{bn^2+cn}/(q)_{n}$; see Theorem \ref{ramanujan asymptotic dilog} of this paper. Berndt \cite[pp.~269--284]{ramanujanIV} has given two proofs of this result, the first developed along the lines of Meinardus \cite{meinardus}, and the second being the one given by McIntosh \cite{mcintosh}. McIntosh's proof also gives additional terms in the asymptotic expansion; see \cite[Equation (5.3)]{mcintosh}. This asymptotic formula involves the dilogarithm which, although a very useful function having an extensive theory replete with beautiful formulas, is still a special function. 

On the other hand, if we multiply \emph{two} series of this form, then as shown below, we get an asymptotic formula which is simpler since it involves only elementary functions.

\begin{theorem}\label{asymptotic of product theorem}
Let $a>0$ and $s\in\mathbb{N}$. Let  $z_1$ denote the real positive root of $az^{1/s}+z-1=0$. As $q\to1^{-}$,
\begin{align}\label{prodasy}
	\sum_{n=0}^{\infty}\frac{a^nq^{n^2/(2s)}}{(q)_n}\sum_{n=0}^{\infty}\frac{a^{-ns}	q^{n^2s/2}}{(q)_n}\sim\frac{\sqrt{s}}{1+(s-1)z_1}\exp{\left(-\frac{1}{\log(q)}\left(\frac{\pi^2}{6}+\frac{s}{2}\log^{2}(a)\right)\right)},
\end{align}
\end{theorem}
\begin{proof}
Letting $b=1/(2s)$ and $c=0$ in Theorem \ref{ramanujan asymptotic dilog}, we see that
\begin{align}\label{asym1}
\sum_{n=0}^{\infty}\frac{a^nq^{n^2/(2s)}}{(q)_n}\sim\frac{1}{\sqrt{z_1+(1-z_1)/s}}\exp{\left(-\frac{1}{\log(q)}\left(\textup{Li}_2(az_1^{1/s})+\frac{1}{2s}\log^{2}(z_1)\right)\right)},
\end{align}	
where $z_1$ is the positive root of $az^{1/s}+z-1=0$, and, $z_1+(1-z_1)/s>0$. Similarly, letting $b=s/2$ and $c=0$ in Theorem \ref{ramanujan asymptotic dilog} and replacing $a$ by $a^{-s}$, we have
\begin{align}\label{asym2nd}
	\sum_{n=0}^{\infty}\frac{a^{-ns}q^{n^2s/2}}{(q)_n}\sim\frac{1}{\sqrt{z_2+s(1-z_2)}}\exp{\left(-\frac{1}{\log(q)}\left(\textup{Li}_2(a^{-s}z_2^{s})+\frac{s}{2}\log^{2}(z_2)\right)\right)},
\end{align}	
where $z_2$ is the positive root of 
\begin{equation}\label{2nd equation}
a^{-s}z^{s}+z-1=0,
\end{equation}
and $z_2+s(1-z_2)>0$.

Equation \eqref{2nd equation} implies that $a^{-s}z_2^{s}$ is a positive number satisfying the equation $az^{1/s}+z-1=0$. Therefore, by uniqueness, we must have $z_1=a^{-s}z_2^{s}$. This, along with \eqref{2nd equation}, then implies $z_1+z_2=1$.

Now \eqref{asym1} and \eqref{asym2nd} yield
\begin{align}\label{prodasy1}
	\sum_{n=0}^{\infty}\frac{a^nq^{n^2/(2s)}}{(q)_n}\sum_{n=0}^{\infty}\frac{a^{-ns}	q^{n^2s/2}}{(q)_n}
	&\sim\frac{1}{\sqrt{(z_1+(1-z_1)/s)(z_2+s(1-z_2))}}\nonumber\\
	&\times\exp{\left(\tfrac{-1}{\log(q)}\left(\textup{Li}_2(az_1^{1/s})+\textup{Li}_2(a^{-s}z_2^{s})+\tfrac{1}{2s}\log^{2}(z_1)+\tfrac{s}{2}\log^{2}(z_2)\right)\right)}.
\end{align}
Next, use the functional equation \cite[p.~5, Equation (1.11)]{lewin}
\begin{equation*}
\textup{Li}_2(z)+\textup{Li}_2(1-z)=\frac{\pi^2}{6}-\log(z)\log(1-z)
\end{equation*}
while repeatedly using the facts $z_1=a^{-s}z_2^{s}$ and $z_1+z_2=1$ to obtain
\begin{align}\label{dilog simp}
&\textup{Li}_2(az_1^{1/s})+\textup{Li}_2(a^{-s}z_2^{s})+\frac{1}{2s}\log^{2}(z_1)+\frac{s}{2}\log^{2}(z_2)\nonumber\\
&=\textup{Li}_2(z_2)+\textup{Li}_2(1-z_2)+\frac{1}{2s}\log^{2}\left(a^{-s}z_2^{s}\right)+\frac{s}{2}\log^{2}(z_2)\nonumber\\
&=\frac{\pi^2}{6}-\log(z_2)\log(1-z_2)+\frac{s}{2}\left(\log(z_2)-\log(a)\right)^2+\frac{s}{2}\log^{2}(z_2)\nonumber\\
&=\frac{\pi^2}{6}-s\log(z_2)\left(\log(z_2)-\log(a)\right)+\frac{s}{2}\left(\log(z_2)-\log(a)\right)^2+\frac{s}{2}\log^{2}(z_2)\nonumber\\
&=\frac{\pi^2}{6}+\frac{s}{2}\log^{2}(a).
\end{align}
Finally, observe that
\begin{align}\label{algesim}
(z_1+(1-z_1)/s)(z_2+s(1-z_2))=(1+(s-1)z_1)^2/s.
\end{align}
The required asymptotic now follows from \eqref{prodasy1}, \eqref{dilog simp} and \eqref{algesim}.
\end{proof}
\begin{remark}
Using McIntosh's more general version of Ramanujan's Theorem \ref{ramanujan asymptotic dilog}, namely \cite[Equation (5.3)]{mcintosh}, one may calculate further terms in the asymptotic expansion of the left-hand side of \eqref{prodasy}.
\end{remark}

\begin{remark}\label{e1}
	 Letting $s=2$ in Theorem \ref{asymptotic of product theorem} shows, in particular, that as $q\to1^{-}$,
\begin{align*}
	\sum_{n=0}^{\infty}\frac{a^nq^{{n^2}/4}}{(q)_n}\sum_{m=0}^{\infty}\frac{a^{-2m}q^{m^2}}{(q)_m}\sim\frac{2\sqrt{2}}{4+a^2-a\sqrt{4+a^2}}\exp{\left\{-\frac{1}{\log(q)}\left(\frac{\pi^2}{6}+\log^2(a)\right)\right\}}.
\end{align*}
\end{remark}
\begin{remark}\label{e2}
Using Theorem \ref{ramanujan asymptotic dilog}, one can similarly show that the second term on the left-hand side of \eqref{gmr} with $b=1$ satisfies the following asymptotic formula as $q\to1^{-}$:
\begin{align*}
\sum_{m=0}^{\infty}\frac{a^{-2m-1}q^{m^2+m}}{(q)_m}\sum_{n=0}^{\infty}\frac{a^nq^{{(n+1)^2}/4}}{(q)_n}\sim\sqrt{2}q^{1/4}\frac{(2+a^2-a\sqrt{4+a^2})}{(4+a^2-a\sqrt{4+a^2})}\exp{\left\{-\frac{1}{\log(q)}\left(\frac{\pi^2}{6}+\log^2(a)\right)\right\}}.
\end{align*}
\end{remark}
\begin{remark}
	The case $s=1$ of Theorem \ref{asymptotic of product theorem} was previously obtained by McIntosh \cite[Equation (8)]{mcintosh2}. Indeed, letting $c=0$ in his formula gives this special case of our Theorem \ref{asymptotic of product theorem}.
\end{remark}
\section{Another identity of Ramanujan}\label{another}
On page $26$ of the Lost Notebook \cite{lnb} (see \cite[p.~156, Entry 7.2.4]{aar2}), Ramanujan claimed
\begin{align}\label{gmr2}
\sum_{n=-\infty}^{\infty}a^nq^{n^2/4}\sum_{n=0}^{\infty}\frac{(-a)^nq^{n^2/4}}{(q)_n}-\sum_{n=-\infty}^{\infty}(-a)^nq^{n^2/4}\sum_{n=0}^{\infty}\frac{a^nq^{n^2/4}}{(q)_n}=2q^{1/4}(q)_{\infty}\sum_{m=0}^{\infty}\frac{a^{-2m-1}q^{m^2+m}}{(bq)_m}.
\end{align}
It was proved by Andrews in \cite[pp.~32-33, Theorem 2.4]{yesto}. He also showed \cite[p.~27]{yesto} that Watson's identities \cite[Equations (6), (7)]{watson}
\begin{align}\label{watson}
G(-q)\varphi(q)-G(q)\varphi(-q)&=2qH(q^4)\psi(q^2),\notag\\
H(-q)\varphi(q)+H(q)\varphi(-q)&=2G(q^4)\psi(q^2),
\end{align}
where the functions $G$, $H$, and $\varphi(q)$ are defined in \eqref{rrf defn} and \eqref{varphiq} and 
	$\psi(q)=(q^2;q^2)_{\infty}/(q;q^2)_{\infty}$
follow easily as special cases of \eqref{gmr2}. 

We now show that identities \eqref{gmr} and \eqref{gmr2} together give an identity of Ramanujan from among his forty identities as a special case. This does not seem to have been noticed before. This identity is \cite[p.~9, Entry 3.20]{berndt40}
\begin{align}\label{entry3.20}
G(q)H(-q)+G(-q)H(q)=2(-q^2;q^2)_{\infty}^2.	
\end{align} 
To see this, substitute for the theta function $\sum_{n=-\infty}^{\infty}a^nq^{n^2/4}$ occurring in \eqref{gmr2}, the left-hand side of \eqref{gmr} with $b=1$, and do the same for $\sum_{n=-\infty}^{\infty}(-a)^nq^{n^2/4}$ by letting $b=1$ and replacing $a$ by $-a$ in \eqref{gmr}. Together, this yields, after simplification,
\begin{align*}
	\sum_{n=0}^{\infty}\frac{a^nq^{(n+1)^2/4}}{(q)_n}\sum_{n=0}^{\infty}\frac{(-a)^nq^{n^2/4}}{(q)_n}+\sum_{n=0}^{\infty}\frac{(-a)^nq^{(n+1)^2/4}}{(q)_n}\sum_{n=0}^{\infty}\frac{a^nq^{n^2/4}}{(q)_n}=2q^{1/4}.
\end{align*}
Now let $a=1$ (or $-1$) in the above identity, replace $q$ by $q^{4}$, and then use Rogers' identities \eqref{rogers identities1} and \eqref{rogers identities2}, once as they are, and another time with $q$ replaced by $-q$, to arrive at \eqref{entry3.20}.

\section{Concluding remarks}\label{cr}
The result in Theorem \ref{theorem2}  not only generalizes Ramanujan's \eqref{gmr} but also gives, in particular, its natural proof. That being said, there is still some air of mystery about why Ramanujan thought there should exist such an elegant representation of the Jacobi theta function in terms of generalized Rogers-Ramanujan functions. Asymptotics of both sides of \eqref{gmr} could very well be a plausible answer, which is a possibility considered, also, for some of Ramanujan's identities involving the Rogers-Ramanujan functions \cite[p.~3]{berndt40}. 
	
Indeed, Remarks \ref{e1} and \ref{e2} together imply that as $q\to 1^{-}$,
\begin{align}\label{gmrlhs}
	&\sum_{m=0}^{\infty}\frac{a^{-2m}q^{m^2}}{(q)_m}\sum_{n=0}^{\infty}\frac{a^nq^{{n^2}/4}}{(q)_n}+\sum_{m=0}^{\infty}\frac{a^{-2m-1}q^{m^2+m}}{(q)_m}\sum_{n=0}^{\infty}\frac{a^nq^{{(n+1)^2}/4}}{(q)_n}\nonumber\\
&\sim\sqrt{2}\frac{\left(2+q^{1/4}(2+a^2-a\sqrt{4+a^2})\right)}{(4+a^2-a\sqrt{4+a^2})}\exp{\left\{-\frac{1}{\log(q)}\left(\frac{\pi^2}{6}+\log^2(a)\right)\right\}}\notag\\
&\sim\sqrt{2}\exp{\left\{-\frac{1}{\log(q)}\left(\frac{\pi^2}{6}+\log^2(a)\right)\right\}}.
	\end{align}
Now using \cite[Theorem 2]{mcintosh3}, we infer that as $q\to1^{-}$,
\begin{align}\label{poch asy}
\frac{1}{(q;q)_{\infty}}\sim\sqrt{\frac{-\log(q)}{2\pi}}\exp{\left(-\frac{\pi^2}{6\log(q)}\right)}.
\end{align}
Moreover, by observing that the $n=0$ term on the right-hand side of the transformation formula \cite[p.~38, Equation (2.2.5)]{borweins}
\begin{equation*}
\sum_{n=-\infty}^{\infty}e^{2\pi inx}e^{-\pi n^2/w}=\sqrt{w}\sum_{n=-\infty}^{\infty}e^{-w(n+x)^2\pi}\hspace{7mm}(\textup{Re}(w)>0)
\end{equation*}
dominates the rest, we find that as $w\to\infty$,
  \begin{equation*}
  	\sum_{n=-\infty}^{\infty}e^{2\pi inx}e^{-\pi n^2/w}\sim\sqrt{w}e^{-wx^2\pi},
\end{equation*}
which can be recast in the form
\begin{align}\label{theta asy}
\sum_{n=-\infty}^{\infty}a^nq^{n^2/4}\sim2\sqrt{\frac{\pi}{-\log(q)}}\exp{\left(-\frac{\log^2(a)}{\log(q)}\right)}
\end{align}
as $q\to1^{-}$. 
Combining \eqref{poch asy} and \eqref{theta asy}, we see that as $q\to1^{-}$,
\begin{align}\label{gmr rhs}
	\frac{1}{(q;q)_{\infty}}\sum_{n=-\infty}^{\infty}a^nq^{n^2/4}\sim\sqrt{2}\exp{\left\{-\frac{1}{\log(q)}\left(\frac{\pi^2}{6}+\log^2(a)\right)\right\}},
\end{align}
which is the same asymptotic we obtained in \eqref{gmrlhs}! For someone with Ramanujan's acumen, identifying the asymptotic expression in \eqref{gmrlhs} to be that of the left-hand side of \eqref{gmr rhs} would be child's play. This must have led Ramanujan to consider if the left-hand sides of \eqref{gmrlhs} and \eqref{gmr rhs} are equal, in general, for $|q|<1$. However, the genius of Ramanujan goes beyond this as he accommodates one more variable $b$, leading to \eqref{gmr}, so that what we just inferred through the asymptotic analysis is only the special case $b=1$ of \eqref{gmr}!

Once such a multi-parameter identity is found, it is bound to have interesting corollaries, and as mentioned in the introduction (or in \cite{yesto}), we get \eqref{mre}, \eqref{gh5mock} as its special cases. A companion identity of Ramanujan (this time \emph{without} an additional parameter $b$), namely \eqref{gmr2}, has the identities in \eqref{watson} as its special cases. Also, as shown in Section \ref{another}, \eqref{gmr} and the companion identity \eqref{gmr2} together give yet another of Ramanujan's formulas involving the Rogers-Ramanujan functions $G(q)$ and $H(q)$, that is, \eqref{entry3.20}.

After finding two algebraic relations between $G(q)$ and $H(q)$ in \cite{ramanujan algebraic} (see also \cite[p.~231]{cp}), which were the firsts-of-their-kind results back then, Ramanujan wrote `\emph{Each of these formulae is the simplest of a large class}'. It is often believed that the `\emph{large class}' in Ramanujan's remark is the set of forty identities for $G(q)$ and $H(q)$ found by Ramanujan; see, for example, \cite[p.~179]{huang}. This may very well be true. But in light of the fact that \eqref{gmr} and \eqref{gmr2} lead to some of these identities as corollaries, could it be that Ramanujan is referring to such generalized modular relations when he says `\emph{large class}'? If so,  undertaking the task of searching identities analogous to \eqref{gmr} which yield at least some other identities among Ramanujan's set of forty identities is definitely worthwhile. 

Our Theorem \ref{theorem2} shows that theta function can be dissected as  a sum of $s$ number of products of  generalized Rogers-Ramanujan functions, where $s$ is \emph{any} natural number. While $s=2$ of our result leads to Ramanujan's \eqref{gmr}, the case $s=1$ gives a generalization of the Jacobi triple product identity. For $s>2$, however, one transcends the realms of modular functions. Nevertheless, the identities that result for these higher $s$ are interesting in their own right and deserve further study. 

If we let $s=1, 2$ in Theorem \ref{theorem2}, two distinct behaviors are noted for $a=b=1$ and for $a=b=-1$. While the former is associated with the modular world, the latter dwells in the mock modular. Indeed, for $s=1$ and $a=b=-1$, we get Corollary \ref{mock theta gen} whose special case is an identity between two of Ramanujan's third order mock theta functions whereas  for $s=2$ and $a=b=-1$ leads to the identity \eqref{gh5mock} involving two fifth order mock theta functions, as shown by Andrews \cite[p.~28, Equation (2.3.11)]{yesto}. However, the chain seems to break for the very next value of $s$ in that the functions involved in the special case $s=3$ and $a=b=-1$ do not seem to be related to any of the seventh order mock theta functions.

Let $s>2$ in \eqref{theorem 2 eqn}. After Corollary \ref{s=3case}, we offered some evidence on the possible non-existence of the analogues of Rogers' identities \eqref{rogers identities1} and \eqref{rogers identities2} which could possibly be then used to write \eqref{theorem 2 eqn} with $a=b=1$ in a form similar to \eqref{mre}. In particular, it looks doubtful that the two series we get after letting $a=b=1$ in the first expression on the left-hand side of \eqref{s=3mre}, namely, $\sum_{m=0}^{\infty}q^{3m^2/2}/(q)_m$ and $\sum_{n=0}^{\infty}q^{n^2/6}/(q)_n$, are related in a similar way as the functions $\sum_{m=0}^{\infty}q^{m^2}/(q)_m$ and $\sum_{n=0}^{\infty}q^{{n^2}/4}/(q)_n$ are related through \eqref{rogers identities1}. We offer one more piece of evidence below to show why the possibility of there being a relation between the series $\sum_{m=0}^{\infty}q^{3m^2/2}/(q)_m$ and  the sum\footnote{Note that here we have replaced $q$ by $q^6$ in $\sum_{n=0}^{\infty}q^{n^2/6}/(q)_n$ which is analogous to how $q$ was replaced by $q^4$ in $\sum_{n=0}^{\infty}q^{n^2/4}/(q)_n$ to get one of the two series in \eqref{rogers identities1}.} $\sum_{n=0}^{\infty}q^{n^2}/(q^6;q^6)_n$ is bleak.

Consider, first, the case of $\sum_{m=0}^{\infty}q^{m^2}/(q)_m$ and $\sum_{n=0}^{\infty}q^{{n^2}}/(q^4;q^4)_n$. We have $u=\frac{1}{2}(\sqrt{5}-1)$ as the positive root of $z^2+z-1=0$. Then using Theorem \ref{ramanujan asymptotic dilog}, we see that as $q\to1^{-}$, 
\begin{align}\label{ri11}
\frac{\sum_{m=0}^{\infty}q^{m^2}/(q)_m}{\sum_{n=0}^{\infty}q^{{n^2}}/(q^4;q^4)_n}\sim\frac{1}{\sqrt{2}}\exp{\left\{-\frac{1}{\log(q)}\left(\textup{Li}_2(1-u)+\log^{2}(u)-\frac{1}{4}\textup{Li}_2(u)-\frac{1}{16}\log^{2}(1-u)\right)\right\}},
\end{align}
The well-known evaluations of the dilogarithm at $u$ and $1-u$ \cite[p.~7, Equation (1.20)]{lewin} give
\begin{equation}\label{dilogi1}
\textup{Li}_2(1-u)-\frac{1}{4}\textup{Li}_2(u)+\log^{2}(u)-\frac{1}{16}\log^{2}(1-u)=\frac{\pi^2}{24}.
\end{equation}
But from \cite[Theorem 2]{mcintosh2}, one has
\begin{equation}\label{ri12}
(-q^2;q^2)_{\infty}=\frac{1}{(q^2;q^4)_{\infty}}\sim\frac{1}{\sqrt{2}}\exp{\left\{-\frac{\pi^2}{24\log(q)}\right\}}.
\end{equation}
From \eqref{ri11}, \eqref{dilogi1} and \eqref{ri12}, we have derived Rogers' identity \eqref{rogers identities1} asymptotically as $q\to1^{-}$.

Analogously, Theorem \ref{ramanujan asymptotic dilog} can be used to show that if $v$ is the positive root of $z^3+z-1=0$, then as $q\to1^{-}$,
\begin{align}\label{3case}
	\frac{\sum_{m=0}^{\infty}q^{\frac{3}{2}m^2}/(q)_m}{\sum_{n=0}^{\infty}q^{{n^2}}/(q^6;q^6)_n}&\sim\frac{1}{\sqrt{3}}\exp{\left\{-\frac{1}{\log(q)}\left(\textup{Li}_2(1-v)+\frac{3}{2}\log^{2}(v)-\frac{1}{6}\textup{Li}_2(v)-\frac{1}{36}\log^{2}(1-v)\right)\right\}}\notag\\
	&\sim\frac{1}{\sqrt{3}}\exp{\left\{-\frac{1}{\log(q)}\left(\textup{Li}_2(1-v)+\frac{5}{4}\log^{2}(v)-\frac{1}{6}\textup{Li}_2(v)\right)\right\}}.
\end{align}
Now one needs an identity to simplify the expression in the parantheses of \eqref{3case}. It is highly unlikely that an identity yielding something similar  to \eqref{dilogi1} exists. The \texttt{FindIntegerNullVector} command in \emph{Mathematica} suggests\footnote{We are indebted to James McLaughlin \cite{mclaughlin} who showed us this result.} the formula
\begin{align*}
6\textup{Li}_2(v)+36\textup{Li}_2(v^2)-30\textup{Li}_2(v^3)-18\textup{Li}_2(v^4)=0,
\end{align*}
which, using the identity \cite[p.~6, Equation (1.15)]{lewin} $\textup{Li}_2(x^2)=2\left(\textup{Li}_2(x)+\textup{Li}_2(-x)\right)$, reduces to 
\begin{equation*}
	6\textup{Li}_2(v)-30\textup{Li}_2(1-v)-36\textup{Li}_2(-v^2)=\pi^2.
\end{equation*}
Using the standard identities for dilogarithms \cite[pp.~4-5, Equations (1.11), (1.12)]{lewin}, it is easy to show that the above identity is actually true. However, employing it in \eqref{3case} does not reduce the expression in parantheses to a constant multiple of $\pi^2$ since it still \emph{does} depend on $v$. This suggests that it is unlikely that there is a relation between $\sum_{m=0}^{\infty}q^{\frac{3}{2}m^2}/(q)_m$ and $\sum_{n=0}^{\infty}q^{{n^2}}/(q^6;q^6)_n$ similar to \eqref{rogers identities1}. \\

We conclude this section by mentioning problems for future work.\\

\noindent
(1) Does equation \eqref{theorem 2 eqn} represent a physical phenomena? Observe that the $s=2$ case of \eqref{theorem 2 eqn}, that is, \eqref{gmr} along with its companion \eqref{gmr2} vaguely resemble a pair of inversion formulas. Does \eqref{theorem 2 eqn cor2} hint at some homogeneous system? As vague as they sound, these questions are motivated by the fact that the Rogers-Ramanujan functions and Rogers-Ramanujan identities turn up in umpteen areas of Mathematics and Science; see, for example,  \cite{baxter}, \cite{ono-duke}. \\

\noindent
(2)  Can the companion identity \eqref{gmr2} be extended in a similar manner as \eqref{gmr} was generalized in this paper to \eqref{theorem 2 eqn}?\\

\noindent
(3) Do there exist analogues of \eqref{gmr}, more generally, of \eqref{theorem 2 eqn}, where one has sum of products of \emph{three or more} generalized Rogers-Ramanujan functions? This question stems from the fact that there exist identities for sums or differences of products of quadruples of Rogers-Ramanujan functions $G(q)$, $H(q)$. See Robin's thesis \cite[p.~16, Equation (1.27)]{robins}, a recent paper of Channabasavayya, Keerthana and Ranganatha \cite[Theorem 3.1]{ranganatha}, and another recent paper of of Gugg \cite[Theorem 4.1]{gugg}. In light of the fact that \eqref{mre} was a special case of \eqref{gmr}, it is natural to ask if these new relations involving triples or quadruples also admit generalizations like \eqref{gmr}. Perhaps an asymptotic analysis such as the one done at the beginning of this section is the best bet to search for such results. It would be interesting to see what we get on the right-hand sides of such results, provided, of course, they exist.\\
	
\noindent	
(4) Let $a=b=1$ in the series $\sum_{m=0}^{\infty}a^{-sm}q^{sm(m-1)/2+m(k+s/2)}/(bq)_{m}$ occurring in Theorem \ref{theorem2}, where $s\in\mathbb{N}$ and $0\leq k\leq s-1$. Note that the special case $s$ even and $k\leq s/2$ of the resulting series is contained in the series considered by Bressoud, Santos and Mondek \cite[Theorem 1]{bressoud2}. As shown by them, three different restricted partitions have the latter as their generating function. Hence studying the remaining cases of our aforementioned series, that is, those not falling under the purview of the series of Bressoud, Santos and Mondek so as to extract partition-theoretic information from seems to be a worthwhile task.


\begin{center}
\textbf{Acknowledgements}
\end{center}
The authors would like to thank George E. Andrews, Richard McIntosh and James McLaughlin for helpful discussions. The first author is supported by the Swarnajayanti Fellowship grant SB/SJF/2021-22/08 of ANRF (Government of India). He sincerely thanks the agency for the support. The second author is a Sabarmati Bridge Fellow at IIT Gandhinagar. He thanks the institute for the support.

\end{document}